\documentclass[12pt]{article}
\usepackage[utf8]{inputenc}
\usepackage{amsmath, amssymb, amsthm, enumerate, booktabs, accents, framed, graphicx, array, multirow, mathrsfs, cases, euscript, bm, multicol, subcaption}
\usepackage[backend=biber,style=numeric,sorting=nty,giveninits=true,maxbibnames=99,doi=true,url=false,isbn=false]{biblatex}
\usepackage{needspace}

\usepackage[colorlinks,citecolor=blue]{hyperref}
\usepackage[dvipsnames,svgnames,table]{xcolor}
\usepackage{a4wide}
\usepackage{microtype}

\usepackage{tikz}
\usetikzlibrary{arrows.meta,calc,bending}

\definecolor{figgreen}{RGB}{0,112,76}
\definecolor{figred}{RGB}{184,48,38}
\definecolor{figblue}{RGB}{0,92,170}
\definecolor{figyellow}{RGB}{150,110,0}

\newcommand{\bW}{0.95}

\newcommand{\rowGap}{1.60}
\newcommand{\padA}{0.40}
\newcommand{\inSep}{0.45}
\newcommand{\grGap}{0.85}
\newcommand{\padB}{0.33}

\tikzset{
  fgdot/.style   ={circle,draw=none,fill,inner sep=0pt,minimum size=4.6pt},
  gdot/.style    ={fgdot,fill=figgreen},
  rdot/.style    ={fgdot,fill=figred},
  kdot/.style    ={fgdot,fill=black},
  fgbox/.style   ={draw=black,line width=0.8pt},
  fgline/.style  ={line width=1.15pt,line cap=round},
  fgarr/.style   ={fgline,-{Stealth[length=4.6pt,width=4.4pt,bend]},shorten >=1.6pt},
  gedge/.style   ={fgarr,figgreen},
  redge/.style   ={fgarr,figred,dash pattern=on 4pt off 2pt},
  kedge/.style   ={fgarr,black,line width=1.0pt},
  bedge/.style   ={fgarr,figblue,dash pattern=on 1.4pt off 1.2pt},
  yedge/.style   ={fgarr,figyellow,dash pattern=on 5pt off 1.5pt on 1pt off 1.5pt},
  glab/.style    ={inner sep=1pt,fill=white},
  dts/.style     ={inner sep=1pt,font=\bfseries},
}
\newcommand{\dt}{$\cdots$}

\pgfdeclarelayer{fglabels}
\pgfsetlayers{main,fglabels}
\newcommand{\toplabel}[3]{\begin{pgfonlayer}{fglabels}\node[glab] at (#1,#2) {#3};\end{pgfonlayer}}

\newcommand{\grbox}[4]{%
  \pgfmathsetmacro{\bhh}{(\rowGap+2*\padA)/2}%
  \pgfmathsetmacro{\bww}{\bW/2}%
  \pgfmathsetmacro{\ytop}{#3+\rowGap/2}%
  \pgfmathsetmacro{\ybot}{#3-\rowGap/2}%
  \draw[fgbox] (#2-\bww,#3-\bhh) rectangle (#2+\bww,#3+\bhh);
  \toplabel{#2}{#3}{#4}%
  \node[gdot] (#1i1) at (#2-\bww,\ytop) {};
  \node[gdot] (#1o1) at (#2+\bww,\ytop) {};
  \node[rdot] (#1i2) at (#2-\bww,\ybot) {};
  \node[rdot] (#1o2) at (#2+\bww,\ybot) {};
}

\newcommand{\nbox}[6]{%
  \pgfmathsetmacro{\nk}{#4+#5}%
  \pgfmathsetmacro{\htot}{(#4-1)*\inSep+(#5-1)*\inSep+\grGap}%
  \pgfmathsetmacro{\bhh}{(\htot+2*\padB)/2}%
  \pgfmathsetmacro{\bww}{\bW/2}%
  \pgfmathsetmacro{\ylab}{#3+\htot/2-(#4-1)*\inSep-\grGap/2}%
  \draw[fgbox] (#2-\bww,#3-\bhh) rectangle (#2+\bww,#3+\bhh);
  \toplabel{#2}{\ylab}{#6}%
  \foreach \i in {1,...,#4}{%
    \pgfmathsetmacro{\yy}{#3+\htot/2-(\i-1)*\inSep}%
    \node[kdot] (#1i\i) at (#2-\bww,\yy) {};
    \node[kdot] (#1o\i) at (#2+\bww,\yy) {};
  }%
  \foreach \i in {1,...,#5}{%
    \pgfmathsetmacro{\yy}{#3-\htot/2+(#5-\i)*\inSep}%
    \pgfmathtruncatemacro{\j}{#4+\i}%
    \node[kdot] (#1i\j) at (#2-\bww,\yy) {};
    \node[kdot] (#1o\j) at (#2+\bww,\yy) {};
  }%
}
\newcommand{\nboxx}[5]{\nbox{#1}{#2}{#3}{#4}{#5}{}}

\newcommand{\uparcs}[5][gedge]{%
  \draw[#1] (#2) .. controls ($(#2)+(#5,#4)$) and ($(#3)+(-#5,#4)$) .. (#3);}
\newcommand{\dnarcs}[5][redge]{%
  \draw[#1] (#2) .. controls ($(#2)+(#5,-#4)$) and ($(#3)+(-#5,-#4)$) .. (#3);}

\newcommand{\rgbox}[4]{%
  \pgfmathsetmacro{\bhh}{(\rowGap+2*\padA)/2}%
  \pgfmathsetmacro{\bww}{\bW/2}%
  \pgfmathsetmacro{\ytop}{#3+\rowGap/2}%
  \pgfmathsetmacro{\ybot}{#3-\rowGap/2}%
  \draw[fgbox] (#2-\bww,#3-\bhh) rectangle (#2+\bww,#3+\bhh);
  \toplabel{#2}{#3}{#4}%
  \node[rdot] (#1i1) at (#2-\bww,\ytop) {};
  \node[rdot] (#1o1) at (#2+\bww,\ytop) {};
  \node[gdot] (#1i2) at (#2-\bww,\ybot) {};
  \node[gdot] (#1o2) at (#2+\bww,\ybot) {};
}

\newsavebox{\ffbox}
\newlength{\ffW}\newlength{\ffH}\newlength{\ffT}
\newcommand{\fitfig}[3]{%
  \setlength{\ffW}{#1}\setlength{\ffH}{#2}%
  \sbox{\ffbox}{#3}%
  \setlength{\ffT}{\dimexpr\ht\ffbox+\dp\ffbox\relax}%
  \pgfmathsetmacro{\ffs}{min(\ffW/\wd\ffbox,\ffH/\ffT)}%
  \begin{minipage}[b][\ffH][c]{\ffW}%
    \centering\scalebox{\ffs}{\usebox{\ffbox}}%
  \end{minipage}%
}

\allowdisplaybreaks[4]

\numberwithin{equation}{section}

\theoremstyle{plain}
\newtheorem{theorem}{Theorem}[section]
\newtheorem{lemma}[theorem]{Lemma}
\newtheorem{corollary}[theorem]{Corollary}
\newtheorem{proposition}[theorem]{Proposition}

\newtheorem{definition}[theorem]{Definition}
\newtheorem{remark}[theorem]{Remark}
\newtheorem*{theorem*}{Theorem}

\def\bE{\mathbb E}
\def\bP{\mathbb P}
\def\bN{\mathbb N}

\def\bC{\mathbb C}

\def\cP{\mathcal P}

\def\id{\mathrm {Id}}
\def\Tr{\mathrm {Tr}}
\def\tr{\mathrm {tr}}

\def\Wg{\mathrm {Wg}}

\title{Operator Norm Bounds for Multi-leg Matrix Tensors and Applications to Random Matrix Theory}
\author{
Beno\^\i{}t Collins\thanks{Department of Mathematics, Kyoto University, Japan. Email: collins@math.kyoto-u.ac.jp}
\, and
Wangjun Yuan\thanks{Department of Mathematics, Southern University of Science and Technology, Shenzhen, China. Email: ywangjun@connect.hku.hk}
}
\date{\today}

\begin{document}

\maketitle

\begin{abstract}
We study the extremal values of multi-leg traces of matrix tensors under operator norm constraints. A graphical representation gives upper and lower bounds expressed as powers of the common tensor-factor dimension. The lower bounds are attained by matrices that permute tensor factors and are exact within this family. We prove the bounds first for two tensor factors and then for an arbitrary number. Leaving some indices uncontracted produces matrices for which we obtain both moment and operator norm bounds; a single choice of coefficient matrices attains the lower bounds for every positive integer moment. We also obtain exact scalar and operator norm maxima in special cases with sufficiently many tensor factors sharing the same cyclic contraction. As an application, the operator norm bounds yield a uniform comparison between Ginibre products and their free circular counterparts with growing matrix coefficients.
\end{abstract}

\noindent{\bf Mathematics Subject Classification 2020 (MSC2020):} Primary 15A60; Secondary 15B52, 46M05, 46L54, 47A30, 47C15, 47N30, 60B20.

\medskip


\section{Introduction}

The integration of random unitary matrices with respect to the Haar measure plays a fundamental role in random matrix theory, operator algebras, and free probability. A central algebraic tool for evaluating these integrals is the Weingarten calculus \cite{MR4415894,MR4680355}. For two $m$-tuples of matrices $A_1, \ldots, A_m$ and $B_1, \ldots, B_m$ in $M_N(\bC)$, the expectation of their alternating product intertwined with a Haar-distributed unitary matrix $U$ is given by a sum over permutations (see \cite[Corollary~2.4]{CollinsSnady2006}):
$$ \bE \left[ \Tr \left( A_1 U B_1 U^* \cdots A_m U B_m U^* \right) \right] = \sum_{\substack{\sigma, \tau, \rho \in \cP([m]) \\ \sigma\tau\rho = (1, \ldots, m)}} \Tr_\sigma(A_1, \ldots, A_m) \Tr_\tau(B_1, \ldots, B_m) \Wg(\rho, N), $$
where $\cP([m])$ is the set of all permutations on $[m] = \{1,2,\ldots,m\}$, and $\Tr_\sigma(A_1, \ldots, A_m)$ is defined as the product over the cycles of $\sigma$ of the traces of the cyclically ordered products of the matrices. This framework has been used to obtain operator norm estimates for polynomials in Haar unitary matrices \cite{BordenaveCollins2023} and strong asymptotic freeness \cite{MR4756991}.

More recently, there has been a renewed interest in extending these models to the tensor setting, where the integral is taken over $U = U_1 \otimes \dots \otimes U_k$, with each $U_i$ being an independent Haar-distributed random matrix in $M_N(\bC)$ \cite{MR4735823}. Such multi-fold tensor models arise naturally in various contexts, including the spectral distribution of large dimensional tensor products \cite{MR4460273}. In this setting, the integrations yield combinations of fundamental multi-linear tensor invariants of the form:
\begin{align} \label{eq:tensor_invariant}
    (\Tr_{\sigma_1} \otimes \ldots \otimes \Tr_{\sigma_k} ) ( A_1, \ldots, A_m ),
\end{align}
where $A_1, \ldots, A_m \in M_N(\bC)^{\otimes k}$.

In this paper, we turn our attention to controlling the extremal values of these tensor invariants \eqref{eq:tensor_invariant} under the constraint that the operator norms of the coefficient matrices satisfy $\|A_i\| \le 1$. If $k=1$, the solution to this problem is trivial: the maximum is achieved when $A_i = I_N$ for all $i$, yielding exactly $N^{\# \text{cycles}(\sigma_1)}$. However, evaluating these quantities becomes deeply combinatorial and highly non-trivial as soon as the number of legs $k \ge 2$. There are at least three primary motivations to study these operator norm bounds:

\begin{enumerate}
    \item \textbf{Operator Algebras and Free Probability:} In the two-leg case ($k=2$), bounds of this type are motivated by questions in random matrix approaches to operator algebras, including Hayes' reduction of the Peterson--Thom conjecture to a strong-convergence problem \cite{MR4448584}. The conjecture has since been resolved through the strong-convergence results of Belinschi--Capitaine and Bordenave--Collins \cite{BelinschiCapitaine2022,BordenaveCollins2023}. Our related two-leg problem concerns uniform norm bounds for contractions indexed by arbitrary permutations $\sigma_1,\sigma_2$.
    \item \textbf{Tensor Models and Invariant Theory:} Tensor invariant theory and tensor analogues of free probability have developed substantially \cite{MR4735823,MR4604905,CGL2024,BonninBordenave2024,NechitaPark2025}. These approaches use permutation-indexed invariants, tensor cumulants, or graphical notions of freeness. Our contribution is a set of finite-dimensional operator norm bounds for the multi-leg contractions in \eqref{eq:tensor_invariant}.
    \item \textbf{Quantum Information Theory:} Bounds on these tensor invariants directly capture the presence and limitations of entanglement across multiple tensor subsystems \mbox{\cite{CN2010,MR4604905}}. Our bounds compare explicit permutation tensors, defined in \eqref{eq-permutation-tensor}, with general coefficients that may be entangled across the tensor legs.
\end{enumerate}

To tackle this, we develop a comprehensive graphical calculus that translates the multilinear tensor structure into the analysis of colored directed graphs. Throughout this paper, for the sake of clarity, we carefully detail the proofs for the two-leg case ($k=2$) before generalizing to an arbitrary number of legs. We also extend our framework beyond full permutations to arbitrary partial permutations (Section~\ref{sec:graph-partial}), so that the resulting multi-leg matrices and their moments admit the same graphical description.

Let us summarize the main results of this paper. Our analysis relies on interpreting the multi-linear tensor structure as a colored directed graph $G_{\sigma_1, \ldots, \sigma_k}$. In this graph, each matrix $A_i$ is represented as a ``rectangle'' with $k$ in-vertices and $k$ out-vertices, and the permutations $\sigma_j$ define the external directed edges between different rectangles. By introducing auxiliary internal edges (referred to as ``blue edges'') that pair the in-vertices and out-vertices within each rectangle, we define our fundamental combinatorial invariant: let $M(\sigma_1, \ldots, \sigma_k)$ denote the maximal number of directed cycles in $G_{\sigma_1, \ldots, \sigma_k}$ over all possible valid connections of these internal blue edges, and let $C(\sigma_1,\ldots,\sigma_k)$ be the minimum number of colored edges that must be removed, retaining all rectangles, so that no blue-edge configuration of the remaining graph contains a directed cycle. The graph and both invariants are defined in Section~\ref{sec:graph-multi leg}, particularly Definition~\ref{def-cycle-cut}. Our main result gives the following two bounds.

\begin{theorem*}[cf. Corollary \ref{Coro-main-2 leg} and Theorem \ref{Thm-main-multi}]
For $m,k,N\in\bN$ and permutations $\sigma_1,\ldots,\sigma_k\in\cP([m])$,
\begin{align*}
N^{M(\sigma_1,\ldots,\sigma_k)}
\le \max_{\|A_1\|,\ldots,\|A_m\|\le1}
\left|(\Tr_{\sigma_1}\otimes\cdots\otimes\Tr_{\sigma_k})(A_1,\ldots,A_m)\right|
\le N^{C(\sigma_1,\ldots,\sigma_k)}.
\end{align*}
Here $A_i\in M_N(\bC)^{\otimes k}$. The lower bound is attained by permutation tensors.
\end{theorem*}

For partial permutations, we define $M(\sigma_1,\ldots,\sigma_k)$ and $C(\sigma_1,\ldots,\sigma_k)$ in the same way, and let $q=\sum_{j=1}^k(m-|D(\sigma_j)|)$, where $D(\sigma_j)$ is the domain of $\sigma_j$ (Section~\ref{sec:graph-partial}). The resulting matrix $Y$ acts on a space of dimension $N^q$.
\begin{theorem*}[cf. Theorem \ref{Thm-main-matrix} and Corollary \ref{Coro-operator norm}]
For $m,k,N\in\bN$ and partial permutations $\sigma_1,\ldots,\sigma_k\in\cP'([m])$,
\begin{align*}
N^{2pM(\sigma_1,\ldots,\sigma_k)+q} & \le \max_{\|A_1\|,\ldots,\|A_m\|\le1} \Tr((YY^*)^p) \le N^{2pC(\sigma_1,\ldots,\sigma_k)+q},\qquad \forall p\in\bN.
\end{align*}
Moreover,
\begin{align*}
    N^{M(\sigma_1,\ldots,\sigma_k)}&\le \max_{\|A_1\|,\ldots,\|A_m\|\le1} \|Y\|\le N^{C(\sigma_1,\ldots,\sigma_k)}.
\end{align*}
Here $A_i\in M_N(\bC)^{\otimes k}$.
\end{theorem*}

One tuple of permutation tensors attains all the lower bounds above. In general, our results give separate lower and upper bounds. Exact maxima are established for the families in Corollaries~\ref{Coro-multi-backward} and~\ref{Coro-open-chain-equality}.

The upper bound for scalar contractions follows from Cauchy--Schwarz and a Hilbert--Schmidt norm identity for simple partial graphs (Definition~\ref{def-cycle-cut}). The lower bound follows by assigning one free index to each directed cycle of a blue-edge configuration. See Section \ref{sec:lower}.

As an application, Section~\ref{sec:application} compares Gaussian matrix products with their free circular counterparts, allowing coefficient matrices in $M_{n_1}(\bC)\otimes M_{n_2}(\bC)$. For $n_1=N^{d_1}$ and $n_2=N^{d_2}$ with $d_1>d_2$, we obtain a uniform operator norm error of order $n_2/n_1$. The proof applies the operator norm upper bound to the crossing terms of the Wick expansion.

\subsection*{Organization of the paper}
Section~\ref{sec:graph} defines the graphical notation and its coordinate evaluation, including partial permutations and moment graphs. Section~\ref{sec:results} states the scalar, moment, and operator norm bounds. Sections~\ref{sec:upper} and~\ref{sec:lower} prove the upper and lower bounds respectively for two legs. Section~\ref{sec:2legs} combines them and proves the reduction from contractions to unitaries. Section~\ref{sec:multi-legs} extends the scalar bounds to any number of legs. Section~\ref{sec:matrix} proves the moment bounds, and Section~\ref{sec:proof-coro} proves the remaining corollaries, including the operator norm bounds and the special-case exact evaluations. Section~\ref{sec:application} applies these results to a uniform Ginibre comparison with matrix coefficients, using backward edges to estimate the crossing terms of the Wick expansion.

\subsection*{Acknowledgements}

BC was supported by JSPS Grant-in-Aid Scientific Research (A) no. 25H00593, and Challenging Research (Exploratory) no. 23K17299.
WY was supported by Guangdong Basic and Applied Basic Research Foundation (No. 2026A1515030040), and National Natural Science Foundation of China (No. 12501183), and a Grant of the Department of Science and Technology of Guangdong Province (No. 2024QN11X161).

\section{Graphical notation} \label{sec:graph}

In order to present our main results, we introduce graphical terminology. The representation by rectangles and colored edges originates in the graphical calculus of Collins and Nechita \cite{CN2010}; through Figure~\ref{figure-example'}, we follow its tensor formulation in Collins, Gurau, and Lionni \cite{MR4604905}. The auxiliary blue edges encode the permutation tensors of \eqref{eq-permutation-tensor} used for lower bounds, and the yellow edges introduced in Section~\ref{sec:graph-moment} encode moment contractions. In colored diagrams, green edges are solid, red edges dashed, blue edges dotted, and yellow edges dash-dotted.

We use the standard basis of $(\bC^N)^{\otimes k}$. Write
$(A_i)_{t_i^1\cdots t_i^k;s_i^1\cdots s_i^k}$ for the tensor matrix entries of $A_i$, where all indices range over $[N]$. The row indices $t$ label in-vertices and the column indices $s$ label out-vertices. For permutations, the index contraction is
\begin{equation}\label{eq-coordinate-trace}
(\Tr_{\sigma_1}\otimes\cdots\otimes\Tr_{\sigma_k})(A_1,\ldots,A_m)
=\sum_{\mathbf t,\mathbf s}
\prod_{i=1}^m (A_i)_{t_i^1\cdots t_i^k;s_i^1\cdots s_i^k}
\prod_{j=1}^k\prod_{i=1}^m\delta_{s_i^j,t_{\sigma_j(i)}^j},
\end{equation}
where $\sum_{\mathbf t,\mathbf s}$ sums over $t_i^l$ and $s_i^l$ for all $l \in [k]$ and $i \in [m]$.
Thus each colored edge identifies its two endpoint indices.
In this section, we first introduce the graph corresponding to the partial trace for the case of 2 legs in Section \ref{sec:graph-2 leg}, and then the graph for the multiple leg version in Section \ref{sec:graph-multi leg}.
Then in Section \ref{sec:graph-partial}, we extend the graphs for partial permutations, which play a key role in our proofs.
Lastly, we construct the graph for the moments of the partial trace of partial permutations via its partial graph in Section \ref{sec:graph-moment}.

\subsection{Graph for 2 legs} \label{sec:graph-2 leg}

We start with the case $k=2$ as a warm-up. In this manuscript, we take the stance that a good understanding of the case with two legs will facilitate the statement and the proof of the general case.
Our first goal is to evaluate and bound this contraction: for permutations $\sigma_1, \sigma_2 \in \cP([m])$, we compute the following partial trace:
\begin{align} \label{tr-sigma-tau}
    \left( \Tr_{\sigma_1} \otimes \Tr_{\sigma_2} \right) \left( A_1, \ldots, A_m \right)
\end{align}
and bound its maximum among all unitary matrices $A_1, \ldots, A_m \in M_N (\bC) \otimes M_N(\bC)$.

We use a graph with the notation $G_{\sigma_1,\sigma_2} (A_1,\ldots,A_m)$ for the representation of the partial trace \eqref{tr-sigma-tau}. We use $m$ rectangles for the matrices $A_1,\ldots,A_m$. We connect the rectangles with directed edges according to the permutations $\sigma_1, \sigma_2$.

For the permutation $\sigma_1$, for all $i$, we use a green directed edge to connect the rectangles $A_i$ and $A_{\sigma_1(i)}$ with the orientation from the rectangle $A_i$ to the rectangle $A_{\sigma_1(i)}$. The vertices of the green edge on the rectangle $A_i$ and the rectangle $A_{\sigma_1(i)}$ are called the green \textit{out-vertex} and the green \textit{in-vertex}, respectively. We also call the green edge an \textit{out-edge} of $A_i$ and an \textit{in-edge} of $A_{\sigma_1(i)}$.
We use red directed edges to connect rectangles for any permutation $\sigma_2$, following the same rule.
That is, the red directed edges are from $A_i$ to $A_{\sigma_2(i)}$.
Then, every rectangle has one green in-edge, one green out-edge, one red in-edge, and one red out-edge. Thus, there are four vertices on every rectangle: the green in-vertex, the green out-vertex, the red in-vertex, and the red out-vertex.
Moreover, if we shrink every rectangle to a point, then the connected components of the red and green directed edges correspond to the cycles of the permutations $\sigma_2$ and $\sigma_1$, respectively.
It is obvious that the graph defined above has $m$ red edges and $m$ green edges.
For simplicity, unless otherwise specified, we always draw the out-vertices on the right of the rectangles, and the in-vertices on the left of the rectangles. Unless specified otherwise, the quantities we consider in the sequel do not depend on the position of each rectangle. Without loss of generality, we always list the rectangles $A_1,\ldots,A_m$ from left to right.
Figure \ref{figure-example'} below is an example for $m=4$, $\sigma_1=(1,2,3)(4)$, and $\sigma_2=(1,2,3,4)$.
\begin{figure}[ht]
    \centering
    \fitfig{225.51bp}{113.01bp}{%
    \begin{tikzpicture}
      \useasboundingbox (-0.735,-1.872) rectangle (6.181,1.868);
      \grbox{A}{0}{0}{$A_1$}   \grbox{B}{1.82}{0}{$A_2$}
      \grbox{C}{3.64}{0}{$A_3$} \grbox{D}{5.46}{0}{$A_4$}
      \draw[gedge] (Ao1) -- (Bi1);
      \draw[gedge] (Bo1) -- (Ci1);
      \uparcs[gedge]{Co1}{Ai1}{1.35}{0.95}
      \uparcs[gedge]{Do1}{Di1}{0.95}{0.30}
      \draw[redge] (Ao2) -- (Bi2);
      \draw[redge] (Bo2) -- (Ci2);
      \draw[redge] (Co2) -- (Di2);
      \dnarcs[redge]{Do2}{Ai2}{1.35}{0.95}
    \end{tikzpicture}%
  }
    \caption{Graph $G_{\sigma_1,\sigma_2} (A_1,\ldots,A_m)$ with $m=4$ and $\sigma_1=(1,2,3)(4)$, $\sigma_2=(1,2,3,4)$}
    \label{figure-example'}
\end{figure}

Next, we introduce some auxiliary directed edges for the graph $G_{\sigma_1,\sigma_2} (A_1,\ldots,A_m)$. For all rectangles, we use blue directed edges to pair the in-vertices and out-vertices of the same rectangle. The blue directed edges are arbitrary, with the constraints that their orientations are from the in-vertices to the out-vertices of the same rectangles and that different blue directed edges do not share any common vertices. As every rectangle has two in-vertices and two out-vertices, it has two blue directed edges. There are two different possibilities for the pair of blue directed edges: the vertices of the blue directed edges may have the same color or different colors.
In the following, all graphs refer to the graphs without blue directed edges unless otherwise specified.
Figure \ref{figure-extended graph-example} is an example of a possibility of the blue directed edges in Figure \ref{figure-example'}.

\begin{figure}[ht]
    \centering
    \fitfig{224.51bp}{114.01bp}{%
    \begin{tikzpicture}
      \useasboundingbox (-0.735,-1.872) rectangle (6.181,1.868);
      \grbox{A}{0}{0}{$A_1$}   \grbox{B}{1.82}{0}{$A_2$}
      \grbox{C}{3.64}{0}{$A_3$} \grbox{D}{5.46}{0}{$A_4$}
      \draw[bedge] (Ai1) -- (Ao2);  \draw[bedge] (Ai2) -- (Ao1);
      \draw[bedge] (Bi1) -- (Bo1);  \draw[bedge] (Bi2) -- (Bo2);
      \draw[bedge] (Ci1) -- (Co2);  \draw[bedge] (Ci2) -- (Co1);
      \draw[bedge] (Di1) -- (Do1);  \draw[bedge] (Di2) -- (Do2);
      \draw[gedge] (Ao1) -- (Bi1);
      \draw[gedge] (Bo1) -- (Ci1);
      \uparcs[gedge]{Co1}{Ai1}{1.35}{0.95}
      \uparcs[gedge]{Do1}{Di1}{0.95}{0.30}
      \draw[redge] (Ao2) -- (Bi2);
      \draw[redge] (Bo2) -- (Ci2);
      \draw[redge] (Co2) -- (Di2);
      \dnarcs[redge]{Do2}{Ai2}{1.35}{0.95}
    \end{tikzpicture}%
  }
    \caption{Blue directed edges in graph $G_{\sigma_1,\sigma_2} (A_1,\ldots,A_m)$ with $m=4$ and $\sigma_1=(1,2,3)(4)$, $\sigma_2=(1,2,3,4)$}
    \label{figure-extended graph-example}
\end{figure}

We call $G'_{\sigma_1,\sigma_2} (A_1,\ldots,A_m)$ a \emph{partial graph} if it can be obtained from $G_{\sigma_1,\sigma_2} (A_1,\ldots,A_m)$ by removing some red directed edges, green directed edges, and rectangles. The green directed edges, red directed edges, and rectangles, which are removed from $G_{\sigma_1,\sigma_2} (A_1,\ldots,A_m)$ to obtain the partial graph $G'_{\sigma_1,\sigma_2} (A_1,\ldots,A_m)$, form a partial graph $G''_{\sigma_1,\sigma_2} (A_1,\ldots,A_m)$. The partial graph $G''_{\sigma_1,\sigma_2} (A_1,\ldots,A_m)$ is called the \emph{complement} of $G'_{\sigma_1,\sigma_2} (A_1,\ldots,A_m)$.

\subsection{Graph for multiple legs} \label{sec:graph-multi leg}

Now we introduce the multiple-leg version of the graph.
For permutations $\sigma_1,\ldots,\sigma_k \in \cP([m])$, we can interpret the partial trace \eqref{eq:tensor_invariant} as a graph $G_{\sigma_1,\ldots,\sigma_k} (A_1,\ldots,A_m)$ that is similar to the graph in Section \ref{sec:graph-2 leg}. We sketch the description below.

The graph $G_{\sigma_1,\ldots,\sigma_k} (A_1,\ldots,A_m)$ has $m$ rectangles, $A_1,\ldots,A_m$, which represent the corresponding matrices.
We introduce $k$ colors $col_1, \ldots, col_k$ corresponding to the $k$ permutations $\sigma_1,\ldots,\sigma_k$.
For $1 \le j \le k$, we connect the $m$ directed edges with the color $col_j$ for the permutation $\sigma_j$ in the graph $G_{\sigma_1,\ldots,\sigma_k} (A_1,\ldots,A_m)$ as follows.
For $1 \le i \le m$, we connect a directed edge of color $col_j$ between $A_i$ and $A_{\sigma_j(i)}$ with an orientation from $A_i$ to $A_{\sigma_j(i)}$. For this edge, the vertex on $A_i$ is an out-vertex, while the vertex on $A_{\sigma_j(i)}$ is an in-vertex. Both vertices are colored with color $col_j$.

Thus, each rectangle has $k$ out-vertices and $k$ in-vertices, and the $k$ out-vertices and the $k$ in-vertices are of color $col_1,\ldots,col_k$, respectively. We provide an example in Figure \ref{figure-example-multi-1} for the graph $G_{\sigma_1,\ldots,\sigma_k} (A_1,\ldots,A_m)$.
For practical reasons, we do not represent the colors in the picture.
In the graph $G_{\sigma_1,\ldots,\sigma_k} (A_1,\ldots,A_m)$, we connect the in-vertices and the out-vertices in each rectangle using $k$ blue directed edges with the orientation from in-vertices to out-vertices, such that the blue directed edges do not have any common vertices. An example of blue directed edges for the graph $G_{\sigma_1,\ldots,\sigma_k} (A_1,\ldots,A_m)$ is provided in Figure \ref{figure-example-multi-2}.

\begin{figure}[ht]
    \centering
    \begin{subfigure}{0.49\textwidth}
        \centering
        \fitfig{223.69bp}{103.69bp}{%
    \begin{tikzpicture}
      \useasboundingbox (-0.785,-1.824) rectangle (6.428,1.858);
      \nbox{A}{0}{0}{2}{2}{$A_1$}      \nbox{B}{1.82}{0}{2}{2}{$A_2$}
      \nbox{C}{3.64}{0}{2}{2}{$A_3$}   \nbox{D}{5.46}{0}{2}{2}{$A_4$}
      \foreach \r in {1,2,3,4}{ \draw[kedge] (Ao\r) -- (Bi\r); \draw[kedge] (Bo\r) -- (Ci\r); }
      \foreach \r in {3,4}{ \draw[kedge] (Co\r) -- (Di\r); }
      \uparcs[kedge]{Co1}{Ai1}{0.90}{0.95}
      \uparcs[kedge]{Co2}{Ai2}{1.83}{1.35}
      \uparcs[kedge]{Do1}{Di1}{0.60}{0.90}
      \uparcs[kedge]{Do2}{Di2}{1.45}{1.50}
      \dnarcs[kedge]{Do3}{Ai3}{1.43}{1.05}
      \dnarcs[kedge]{Do4}{Ai4}{1.18}{1.45}
    \end{tikzpicture}%
  }
        \caption{Graph $G_{\sigma_1,\ldots,\sigma_4} (A_1,\ldots,A_4)$}
        \label{figure-example-multi-1}
    \end{subfigure}
    \begin{subfigure}{0.49\textwidth}
        \centering
        \fitfig{222.73bp}{105.13bp}{%
    \begin{tikzpicture}
      \useasboundingbox (-0.785,-1.824) rectangle (6.428,1.858);
      \nbox{A}{0}{0}{2}{2}{$A_1$}      \nbox{B}{1.82}{0}{2}{2}{$A_2$}
      \nbox{C}{3.64}{0}{2}{2}{$A_3$}   \nbox{D}{5.46}{0}{2}{2}{$A_4$}
      \foreach \r in {1,2,3,4}{ \draw[kedge] (Ao\r) -- (Bi\r); \draw[kedge] (Bo\r) -- (Ci\r); }
      \foreach \r in {3,4}{ \draw[kedge] (Co\r) -- (Di\r); }
      \foreach \r in {1,4}{ \draw[bedge] (Ai\r) -- (Ao\r); \draw[bedge] (Ci\r) -- (Co\r); }
      \draw[bedge] (Ai2) -- (Ao3);  \draw[bedge] (Ai3) -- (Ao2);
      \draw[bedge] (Ci2) -- (Co3);  \draw[bedge] (Ci3) -- (Co2);
      \foreach \r in {1,2,3,4}{ \draw[bedge] (Bi\r) -- (Bo\r); \draw[bedge] (Di\r) -- (Do\r); }
      \uparcs[kedge]{Co1}{Ai1}{0.90}{0.95}
      \uparcs[kedge]{Co2}{Ai2}{1.83}{1.35}
      \uparcs[kedge]{Do1}{Di1}{0.60}{0.90}
      \uparcs[kedge]{Do2}{Di2}{1.45}{1.50}
      \dnarcs[kedge]{Do3}{Ai3}{1.43}{1.05}
      \dnarcs[kedge]{Do4}{Ai4}{1.18}{1.45}
    \end{tikzpicture}%
  }
        \caption{Graph $G_{\sigma_1,\ldots,\sigma_4} (A_1,\ldots,A_4)$ with blue directed edges}
        \label{figure-example-multi-2}
    \end{subfigure}
    \caption{Graph $G_{\sigma_1,\ldots,\sigma_4} (A_1,\ldots,A_4)$ with $\sigma_1=\sigma_2=(123)(4)$ and $\sigma_3=\sigma_4=(1234)$}
\end{figure}

For the graph $G_{\sigma_1,\ldots,\sigma_k} (A_1,\ldots,A_m)$, if we remove some directed edges and rectangles, the remaining part $G'_{\sigma_1,\ldots,\sigma_k} (A_1,\ldots,A_m)$ is called a \emph{partial} graph of $G_{\sigma_1,\ldots,\sigma_k} (A_1,\ldots,A_m)$. The directed edges and rectangles that are removed from the graph $G_{\sigma_1,\ldots,\sigma_k} (A_1,\ldots,A_m)$ when obtaining the partial graph $G'_{\sigma_1,\ldots,\sigma_k} (A_1,\ldots,A_m)$ form a partial graph $G''_{\sigma_1,\ldots,\sigma_k} (A_1,\ldots,A_m)$. The partial graph $G''_{\sigma_1,\ldots,\sigma_k} (A_1,\ldots,A_m)$ is called the \emph{complement} of $G'_{\sigma_1,\ldots,\sigma_k} (A_1,\ldots,A_m)$.

\begin{definition}\label{def-cycle-cut}
Let $G_{\sigma_1,\ldots,\sigma_k}(A_1,\ldots,A_m)$ be a graph of permutations $\sigma_1,\ldots,\sigma_k$.
For a blue-edge configuration $\boldsymbol\pi=(\pi_1,\ldots,\pi_m)$, where $\pi_i(j)$ is the color of the out-vertex paired with the in-vertex of color $j$ in rectangle $A_i$, let $M(\sigma_1,\ldots,\sigma_k;\boldsymbol\pi)$ be the number of directed cycles.
Define
\[
M(\sigma_1,\ldots,\sigma_k)=\max_{\boldsymbol\pi\in\cP([k])^m}M(\sigma_1,\ldots,\sigma_k;\boldsymbol\pi).
\]
A partial graph is \emph{simple} if it contains no directed cycle for any blue-edge configuration, and \emph{full} if it retains all rectangles.

Define
\[
C(\sigma_1,\ldots,\sigma_k)=\min_{G'} R_{G'},
\]
where $\min_{G'}$ is taken over all full simple partial graphs $G'$ of $G_{\sigma_1,\ldots,\sigma_k}$, and $R_{G'}$ is the number of colored edges removed from $G_{\sigma_1,\ldots,\sigma_k}$ to obtain $G'$. We also write $C(G)$ for $C(\sigma_1,\ldots,\sigma_k)$ when $G=G_{\sigma_1,\ldots,\sigma_k}$.
\end{definition}
The minimum exists: removing all colored edges leaves a full simple partial graph. To obtain a full simple partial graph, one must cut each cycle of every fixed blue-edge configuration. Thus, $M(\sigma_1,\ldots,\sigma_k)\le C(\sigma_1,\ldots,\sigma_k)$.

The permutation-tensor associated with $\pi\in\cP([k])$ is
\begin{equation}\label{eq-permutation-tensor}
(E_\pi)_{t_1\cdots t_k;s_1\cdots s_k}
=\prod_{j=1}^k\delta_{t_j,s_{\pi(j)}}.
\end{equation}
It is unitary and realizes precisely the blue edges $j\to\pi(j)$. Substituting these entries into \eqref{eq-coordinate-trace} identifies all indices on each directed cycle, with one independent index per cycle. Consequently,
\begin{equation}\label{eq-coordinate-cycles}
(\Tr_{\sigma_1}\otimes\cdots\otimes\Tr_{\sigma_k})(E_{\pi_1},\ldots,E_{\pi_m})
=N^{M(\sigma_1,\ldots,\sigma_k;\boldsymbol\pi)}.
\end{equation}

\subsection{Graphs for partial permutations} \label{sec:graph-partial}

A \emph{partial permutation} of $[m]$ is an injective map $\sigma:D(\sigma)\to[m]$, where $D(\sigma)\subseteq[m]$ is its domain. We denote the set of such maps by $\cP'([m])$. We use $|D(\sigma)|$ for the cardinality of the domain. We would like to remark that a permutation $\sigma$ can also be viewed as a partial permutation with $D(\sigma) = [m]$.

Definition~\ref{def-cycle-cut} extends to partial permutations without change: $M(\sigma_1,\ldots,\sigma_k)$ counts directed cycles (excluding open paths), and $C(\sigma_1,\ldots,\sigma_k)$ minimizes the number of existing colored edges removed while retaining all rectangles and getting rid of directed cycles. We also write $M(\sigma_1,\ldots,\sigma_k;\boldsymbol\pi)$ for the number of directed cycles with a fixed blue-edge configuration $\boldsymbol \pi$ and $M(\sigma_1,\ldots,\sigma_k)$ and $C(\sigma_1,\ldots,\sigma_k)$ for the two invariants of the associated graph $G_{\sigma_1,\ldots,\sigma_k}$.

For partial permutations $\sigma_1,\ldots,\sigma_k \in \cP'([m])$, the row and column indices left uncontracted form the tuples
\[
\alpha=(t_i^j)_{\substack{1\le j\le k\\i\notin\sigma_j(D(\sigma_j))}},
\qquad
\beta=(s_i^j)_{\substack{1\le j\le k\\i\notin D(\sigma_j)}}.
\]
Order each tuple first by color $j$ and then by rectangle $A_i$. Both tuples belong to $[N]^q$, where $q=\sum_{j=1}^k(m-|D(\sigma_j)|)$. For fixed $\alpha,\beta$, define
\begin{equation}\label{eq-coordinate-partial}
Y_{\alpha;\beta}
=\sum_{\substack{s_i^j,t_{\sigma_j(i)}^j\in[N]\\1\le j\le k,\ i\in D(\sigma_j)}}
\prod_{i=1}^m (A_i)_{t_i^1\cdots t_i^k;s_i^1\cdots s_i^k}
\prod_{j=1}^k\prod_{i\in D(\sigma_j)}\delta_{s_i^j,t_{\sigma_j(i)}^j}.
\end{equation}
The sum ranges over both endpoint indices of every retained colored edge: the column indices $s_i^j$ with $i\in D(\sigma_j)$ and the row indices $t_i^j$ with $i\in\sigma_j(D(\sigma_j))$. All other indices are fixed by $\alpha,\beta$. This defines the matrix $Y=(\Tr_{\sigma_1}\otimes\cdots\otimes\Tr_{\sigma_k})(A_1,\ldots,A_m)\in M_{N^q}(\bC)$ and fixes its tensor-factor order. If $q=0$, the tuples are empty and $Y$ is a scalar.

For partial permutations $\sigma_1,\ldots,\sigma_k \in \cP'([m])$, we interpret the matrix $Y$ defined in \eqref{eq-coordinate-partial} as a graph $G_{\sigma_1,\ldots,\sigma_k} (A_1,\ldots,A_m)$ with an idea similar to that of Section \ref{sec:graph-multi leg}. We sketch the description in the following.

The graph $G_{\sigma_1,\ldots,\sigma_k} (A_1,\ldots,A_m)$ has $m$ rectangles, $A_1,\ldots,A_m$, which represent the corresponding matrices.
For all $1 \le j \le k$, we connect the directed edges with the color $col_j$ according to each partial permutation $\sigma_j$.
For $1 \le j \le k$, for the partial permutation $\sigma_j$, the directed edges are from $A_i$ to $A_{\sigma_j(i)}$ for all $i \in D(\sigma_j)$.
For the edge from $A_i$ to $A_{\sigma_j(i)}$, the vertex on $A_i$ is an out-vertex with color $col_j$, while the vertex on $A_{\sigma_j(i)}$ is an in-vertex with color $col_j$.
Note that when $i \notin D(\sigma_j)$, the partial permutation $\sigma_j$ does not provide an out-edge for rectangle $A_i$. In this case, we add an out-vertex  with color $col_j$ on $A_i$.
Similarly, if $i$ is not in the image of $\sigma_j$, then $\sigma_j$ does not provide an in-edge for rectangle $A_i$. We also add an in-vertex with color $col_j$ on $A_i$.
We call the in-vertices and out-vertices we added in this way \emph{open} in-vertices and \emph{open} out-vertices, respectively.

Thus, in the graph $G_{\sigma_1,\ldots,\sigma_k} (A_1,\ldots,A_m)$, each rectangle has $k$ in-vertices and $k$ out-vertices, and the $k$ in-vertices, and likewise the $k$ out-vertices, carry the $k$ distinct colors $col_1,\ldots,col_k$.
The numbers of open in-vertices and open out-vertices both equal $\sum_{j=1}^k (m-|D(\sigma_j)|)$.
Moreover, the number of directed edges in the graph $G_{\sigma_1,\ldots,\sigma_k} (A_1,\ldots,A_m)$ is $\sum_{j=1}^k |D(\sigma_j)|$.
We provide an example in Figure \ref{figure-example-matrix-1} for the graph $G_{\sigma_1,\ldots,\sigma_k} (A_1,\ldots,A_m)$.

In the graph $G_{\sigma_1,\ldots,\sigma_k} (A_1,\ldots,A_m)$, we can connect the in-vertices and the out-vertices in the same rectangle using $k$ blue directed edges whose orientation is from in-vertices to out-vertices, such that the blue directed edges do not have any common vertices.
An example of blue directed edges for the graph $G_{\sigma_1,\ldots,\sigma_k} (A_1,\ldots,A_m)$ is provided in Figure \ref{figure-example-matrix-2}.

\begin{figure}[ht]
    \centering
    \begin{subfigure}{0.48\textwidth}
        \centering
        \fitfig{174.73bp}{89.28bp}{%
    \begin{tikzpicture}
      \useasboundingbox (-0.746,-1.862) rectangle (4.373,1.023);
      \nbox{A}{0}{0}{1}{2}{$A_1$}  \nbox{B}{1.82}{0}{1}{2}{$A_2$}  \nbox{C}{3.64}{0}{1}{2}{$A_3$}
      \foreach \r in {1,2,3}{ \draw[kedge] (Ao\r) -- (Bi\r); \draw[kedge] (Bo\r) -- (Ci\r); }
      \dnarcs[kedge]{Co2}{Ai2}{1.15}{0.95}
      \dnarcs[kedge]{Co3}{Ai3}{1.55}{1.35}
    \end{tikzpicture}%
  }
        \caption{Graph $G_{\sigma_1,\sigma_2,\sigma_3} (A_1,\ldots,A_3)$}
        \label{figure-example-matrix-1}
    \end{subfigure}
    \begin{subfigure}{0.48\textwidth}
        \centering
        \fitfig{174.25bp}{87.84bp}{%
    \begin{tikzpicture}
      \useasboundingbox (-0.746,-1.862) rectangle (4.373,1.023);
      \nbox{A}{0}{0}{1}{2}{$A_1$}  \nbox{B}{1.82}{0}{1}{2}{$A_2$}  \nbox{C}{3.64}{0}{1}{2}{$A_3$}
      \foreach \r in {1,2,3}{ \draw[kedge] (Ao\r) -- (Bi\r); \draw[kedge] (Bo\r) -- (Ci\r); }
      \foreach \r in {1,2,3}{ \draw[bedge] (Ai\r) -- (Ao\r); }
      \draw[bedge] (Bi1) -- (Bo3);  \draw[bedge] (Bi2) -- (Bo1);  \draw[bedge] (Bi3) -- (Bo2);
      \draw[bedge] (Ci1) -- (Co3);  \draw[bedge] (Ci2) -- (Co2);  \draw[bedge] (Ci3) -- (Co1);
      \dnarcs[kedge]{Co2}{Ai2}{1.15}{0.95}
      \dnarcs[kedge]{Co3}{Ai3}{1.55}{1.35}
    \end{tikzpicture}%
  }
        \caption{Graph $G_{\sigma_1,\sigma_2,\sigma_3} (A_1,\ldots,A_3)$ with blue directed edges}
        \label{figure-example-matrix-2}
    \end{subfigure}
    \caption{$k=m=3$, graph $G_{\sigma_1,\sigma_2,\sigma_3} (A_1,\ldots,A_3)$, $\sigma_2=\sigma_3=(123)$ and $\sigma_1(1)=2, \sigma_1(2)=3$}
\end{figure}

Let us make the following remarks about partial graphs:

\begin{remark}
A graph $G_{\sigma_1,\ldots,\sigma_k}(A_1,\ldots,A_m)$ for partial permutations $\sigma_1,\ldots,\sigma_k \in \cP'([m])$ can be viewed as a partial graph of some graph $G_{\sigma_1',\ldots,\sigma_k'} (A_1,\ldots,A_m)$ for some permutations $\sigma_1',\ldots,\sigma_k'$.

However, for a graph $G_{\sigma_1',\ldots,\sigma_k'}(A_1,\ldots,A_m)$ with some permutations $\sigma_1',\ldots,\sigma_k' \in \cP([m])$, its partial graph $G'_{\sigma_1',\ldots,\sigma_k'}(A_1,\ldots,A_m)$ may not be a graph of some partial permutations.
\end{remark}

Lastly, note that the properties of graphs in which we are interested do
not depend on the value of $A_1,\ldots,A_m$. Hence, the graph can also be understood as a map on the space of matrices.

\begin{remark}
For $k,m \in \bN$, for partial permutations $\sigma_1,\ldots,\sigma_k \in \cP'([m])$, the graph $G_{\sigma_1,\ldots,\sigma_k} (A_1,\ldots,A_m)$ is a matrix that belongs to $M_{N^{m-|D(\sigma_1)|}}(\bC) \otimes \ldots \otimes M_{N^{m-|D(\sigma_k)|}}(\bC)$. This correspondence induces a multi-linear mapping
\begin{align*}
    G_{\sigma_1,\ldots,\sigma_k}: \left( \left( M_N(\bC) \right)^{\otimes k} \right)^m
    \longrightarrow M_{N^{m-|D(\sigma_1)|}}(\bC) \otimes \ldots \otimes M_{N^{m-|D(\sigma_k)|}}(\bC).
\end{align*}
In the sequel, we abuse the notation $G_{\sigma_1,\ldots,\sigma_k} (A_1,\ldots,A_m)$ and $G_{\sigma_1,\ldots,\sigma_k}$ if we do not emphasize the matrices $A_1,\ldots,A_m$.
\end{remark}

\subsection{Graphs for moments} \label{sec:graph-moment}

To construct the graph $G^{(p)}_{\sigma_1,\ldots,\sigma_k}(A_1,\ldots,A_m)$ for $\Tr((YY^*)^p)$, we first describe the adjoint. The inverse of a partial permutation $\sigma$ is the injective map $\sigma^{-1}:\sigma(D(\sigma))\to D(\sigma)$. With the row and column tuples of \eqref{eq-coordinate-partial},
\[
(Y^*)_{\beta;\alpha}=\overline{Y_{\alpha;\beta}},
\]
and hence
\begin{align}\label{eq-def-Y*}
Y^*=(\Tr_{\sigma_1^{-1}}\otimes\cdots\otimes\Tr_{\sigma_k^{-1}})(A_1^*,\ldots,A_m^*).
\end{align}
Its graph, denoted by $G^*_{\sigma_1,\ldots,\sigma_k}(A_1^*,\ldots,A_m^*)$, is the mirror image of the graph of $Y$. It is obtained by replacing each $A_i$ by $A_i^*$, exchanging in- and out-vertices, and reversing every colored edge. Thus the $col_j$ edge from $A_i$ to $A_{\sigma_j(i)}$ becomes an edge from $A_{\sigma_j(i)}^*$ to $A_i^*$, and each open in-vertex becomes an open out-vertex, and conversely. We draw the adjoint rectangles in reverse order, keeping in-vertices on the left and out-vertices on the right; this change of drawing does not relabel the matrix slots in \eqref{eq-def-Y*}.

A blue pairing $\pi_i$ in $A_i$ corresponds to the inverse pairing $\pi_i^{-1}$ in $A_i^*$: the edge from the in-vertex with color $col_j$ to the out-vertex with color $col_l$ becomes the edge from the in-vertex with color $col_l$ to the out-vertex with color $col_j$. This gives a bijection of the graphs for partial permutations with blue configurations, so
\[
M(\sigma_1^{-1},\ldots,\sigma_k^{-1})=M(\sigma_1,\ldots,\sigma_k),\qquad C(\sigma_1^{-1},\ldots,\sigma_k^{-1})=C(\sigma_1,\ldots,\sigma_k).
\]
Figures~\ref{figure-example-matrix-2} and~\ref{figure-example-matrix-G*-2} illustrate the mirrored pairings.

\begin{figure}[ht]
    \centering
    \begin{subfigure}{0.48\textwidth}
        \centering
        \fitfig{173.77bp}{84.96bp}{%
    \begin{tikzpicture}
      \useasboundingbox (-0.746,-1.862) rectangle (4.373,1.023);
      \nbox{A}{0}{0}{1}{2}{$A_3^*$}  \nbox{B}{1.82}{0}{1}{2}{$A_2^*$}  \nbox{C}{3.64}{0}{1}{2}{$A_1^*$}
      \foreach \r in {1,2,3}{ \draw[kedge] (Ao\r) -- (Bi\r); \draw[kedge] (Bo\r) -- (Ci\r); }
      \dnarcs[kedge]{Co2}{Ai2}{1.15}{0.95}
      \dnarcs[kedge]{Co3}{Ai3}{1.55}{1.35}
    \end{tikzpicture}%
  }
        \caption{Graph $G^*_{\sigma_1,\sigma_2,\sigma_3}(A_1^*,A_2^*,A_3^*)$}
    \end{subfigure}
    \begin{subfigure}{0.48\textwidth}
        \centering
        \fitfig{172.81bp}{84.48bp}{%
    \begin{tikzpicture}
      \useasboundingbox (-0.746,-1.862) rectangle (4.373,1.023);
      \nbox{A}{0}{0}{1}{2}{$A_3^*$}  \nbox{B}{1.82}{0}{1}{2}{$A_2^*$}  \nbox{C}{3.64}{0}{1}{2}{$A_1^*$}
      \foreach \r in {1,2,3}{ \draw[kedge] (Ao\r) -- (Bi\r); \draw[kedge] (Bo\r) -- (Ci\r); }
      \draw[bedge] (Ai1) -- (Ao3);  \draw[bedge] (Ai2) -- (Ao2);  \draw[bedge] (Ai3) -- (Ao1);
      \draw[bedge] (Bi1) -- (Bo2);  \draw[bedge] (Bi2) -- (Bo3);  \draw[bedge] (Bi3) -- (Bo1);
      \foreach \r in {1,2,3}{ \draw[bedge] (Ci\r) -- (Co\r); }
      \dnarcs[kedge]{Co2}{Ai2}{1.15}{0.95}
      \dnarcs[kedge]{Co3}{Ai3}{1.55}{1.35}
    \end{tikzpicture}%
  }
        \caption{Graph $G^*_{\sigma_1,\sigma_2,\sigma_3}(A_1^*,A_2^*,A_3^*)$ with blue directed edges}
        \label{figure-example-matrix-G*-2}
    \end{subfigure}
    \caption{$k=m=3$, graph $G^*_{\sigma_1,\sigma_2,\sigma_3}(A_1^*,A_2^*,A_3^*)$ with $\sigma_2=\sigma_3=(123)$ and $\sigma_1(1)=2, \sigma_1(2)=3$}

\end{figure}

We now construct the graph for the moment $\Tr ((YY^*)^p)$.
Recall the graphs $G_{\sigma_1,\ldots,\sigma_k} (A_1, \ldots, A_m)$ and $G^*_{\sigma_1,\ldots,\sigma_k} (A_1^*, \ldots, A_m^*)$ for $Y$ and $Y^*$, respectively.
We take copies of the graph $G_{\sigma_1,\ldots,\sigma_k} (A_1, \ldots, A_m)$ and denote the $p$ copies by ${}_rG_{\sigma_1,\ldots,\sigma_k} (A_1, \ldots, A_m)$, where $1 \le r \le p$.
We also take copies of the graph $G^*_{\sigma_1,\ldots,\sigma_k} (A_1^*, \ldots, A_m^*)$ and denote the $p$ copies by ${}_rG^*_{\sigma_1,\ldots,\sigma_k} (A_1^*, \ldots, A_m^*)$ for $1 \le r \le p$.
The graph $G^{(p)}_{\sigma_1,\ldots,\sigma_k} (A_1, \ldots, A_m)$ consists of the graphs ${}_rG_{\sigma_1,\ldots,\sigma_k} (A_1, \ldots, A_m)$ and ${}_rG^*_{\sigma_1,\ldots,\sigma_k} (A_1^*, \ldots, A_m^*)$ for $1 \le r \le p$.
Next, we introduce the directed edges in $G^{(p)}_{\sigma_1,\ldots,\sigma_k} (A_1, \ldots, A_m)$ that connect different copies of ${}_rG_{\sigma_1,\ldots,\sigma_k} (A_1, \ldots, A_m)$ and ${}_rG^*_{\sigma_1,\ldots,\sigma_k} (A_1^*, \ldots, A_m^*)$.
For any $1 \le r \le p$, in the graph ${}_rG_{\sigma_1,\ldots,\sigma_k} (A_1, \ldots, A_m)$, for any open out-vertex on the rectangle $A_i$ with color $col_j$ for some $1 \le i \le m$ and $1 \le j \le k$, we connect it with the open in-vertex of color $col_j$ on the rectangle $A_i^*$ in the graph ${}_rG^*_{\sigma_1,\ldots,\sigma_k} (A_1^*, \ldots, A_m^*)$ using a \emph{yellow} directed edge. The orientation of this yellow directed edge is from the out-vertex to the in-vertex.
Similarly, for any $1 \le r \le p$, in the graph ${}_rG^*_{\sigma_1,\ldots,\sigma_k} (A_1^*, \ldots, A_m^*)$, for any open out-vertex on the rectangle $A_i^*$ with color $col_j$ for some $1 \le i \le m$ and $1 \le j \le k$, we connect it with the open in-vertex of color $col_j$ on the rectangle $A_i$ in the graph ${}_{r+1}G_{\sigma_1,\ldots,\sigma_k} (A_1, \ldots, A_m)$ using a \emph{yellow} directed edge whose orientation is from the out-vertex to the in-vertex. Here, we use the convention ${}_{p+1}G _{\sigma_1,\ldots,\sigma_k} (A_1, \ldots, A_m) = {}_1G_{\sigma_1,\ldots,\sigma_k} (A_1, \ldots, A_m)$.

We provide an example of the moment-graph construction in Figure~\ref{figure-example-G-power}, with $k=3$, $m=2$, and $p=2$ in panels (b)--(c). Figure \ref{figure-example-power-a} is the graph $G_{\sigma_1,\ldots,\sigma_k} (A_1, \ldots, A_m)$, and Figure \ref{figure-example-power-b} is the graph that consists of the duplications of $G_{\sigma_1,\ldots,\sigma_k} (A_1, \ldots, A_m)$ and $G^*_{\sigma_1,\ldots,\sigma_k} (A_1^*, \ldots, A_m^*)$. Figure \ref{figure-example-power-c} is the corresponding graph $G^{(2)}_{\sigma_1,\ldots,\sigma_k} (A_1, \ldots, A_m)$.

\begin{figure}[ht]
    \centering
    \begin{subfigure}{0.3\textwidth}
        \centering
        \fitfig{73.15bp}{56.70bp}{%
    \begin{tikzpicture}
      \useasboundingbox (-0.779,-1.388) rectangle (2.597,1.023);
      \nbox{A}{0}{0}{1}{2}{$A_1$}
      \nbox{B}{1.82}{0}{1}{2}{$A_2$}
      \draw[kedge] (Ao2) -- (Bi2);
      \draw[kedge] (Ao3) -- (Bi3);
      \dnarcs[kedge]{Bo3}{Ai3}{0.92}{1.35}
    \end{tikzpicture}%
  }
        \caption{Graph $G_{\sigma_1,\ldots,\sigma_3}(A_1,A_2)$}
        \label{figure-example-power-a}
    \end{subfigure}
    \begin{subfigure}{0.69\textwidth}
        \centering
        \fitfig{298.22bp}{63.70bp}{%
    \begin{tikzpicture}
      \useasboundingbox (-0.779,-1.879) rectangle (13.965,1.023);
      \nbox{Ra}{0}{0}{1}{2}{$A_1$}      \nbox{Rb}{1.82}{0}{1}{2}{$A_2$}
      \nbox{Rc}{3.64}{0}{1}{2}{$A_2^*$} \nbox{Rd}{5.46}{0}{1}{2}{$A_1^*$}
      \nbox{Re}{7.28}{0}{1}{2}{$A_1$}   \nbox{Rf}{9.10}{0}{1}{2}{$A_2$}
      \nbox{Rg}{10.92}{0}{1}{2}{$A_2^*$}\nbox{Rh}{12.74}{0}{1}{2}{$A_1^*$}
      \foreach \p/\q in {Ra/Rb,Rc/Rd,Re/Rf,Rg/Rh}{%
        \draw[kedge] (\p o2) -- (\q i2);
        \draw[kedge] (\p o3) -- (\q i3);
        \dnarcs[kedge]{\q o3}{\p i3}{0.92}{1.35}
      }

      \node at (0.91,-1.66)  {${}_1G_{\sigma_1,\ldots,\sigma_k}$};
      \node at (4.55,-1.66)  {${}_1G^*_{\sigma_1,\ldots,\sigma_k}$};
      \node at (8.19,-1.66)  {${}_2G_{\sigma_1,\ldots,\sigma_k}$};
      \node at (11.83,-1.66) {${}_2G^*_{\sigma_1,\ldots,\sigma_k}$};
    \end{tikzpicture}%
  }
        \caption{Graphs ${}_1G_{\sigma_1,\ldots,\sigma_3}(A_1,A_2)$, $\ldots, {}_2G^*_{\sigma_1,\ldots,\sigma_3}(A_1^*,A_2^*)$}
        \label{figure-example-power-b}
    \end{subfigure}
    \begin{subfigure}{\textwidth}
        \centering
        \fitfig{450bp}{132bp}{%
    \begin{tikzpicture}
      \nbox{Ra}{0}{0}{1}{2}{$A_1$}       \nbox{Rb}{1.82}{0}{1}{2}{$A_2$}
      \nbox{Rc}{3.64}{0}{1}{2}{$A_2^*$} \nbox{Rd}{5.46}{0}{1}{2}{$A_1^*$}
      \nbox{Re}{7.28}{0}{1}{2}{$A_1$}   \nbox{Rf}{9.10}{0}{1}{2}{$A_2$}
      \nbox{Rg}{10.92}{0}{1}{2}{$A_2^*$}\nbox{Rh}{12.74}{0}{1}{2}{$A_1^*$}
      \foreach \p/\q in {Ra/Rb,Rc/Rd,Re/Rf,Rg/Rh}{%
        \draw[kedge] (\p o2) -- (\q i2);
        \draw[kedge] (\p o3) -- (\q i3);
        \dnarcs[kedge]{\q o3}{\p i3}{0.92}{1.35}
      }
      \foreach \p/\q in {Rb/Rc,Rf/Rg}{%
        \draw[yedge] (\p o1) -- (\q i1);
        \draw[yedge] (\p o2) -- (\q i2);
      }
      \uparcs[yedge]{Rao1}{Rdi1}{1.05}{0.55}
      \uparcs[yedge]{Reo1}{Rhi1}{1.05}{0.55}
      \draw[yedge] (Rdo1) -- (Rei1);
      \draw[yedge] (Rdo2) -- (Rei2);
      \uparcs[yedge]{Rco1}{Rfi1}{1.80}{0.55}
      \uparcs[yedge]{Rgo1}{Rbi1}{2.65}{0.85}
      \uparcs[yedge]{Rho1}{Rai1}{3.50}{1.20}
      \dnarcs[yedge]{Rho2}{Rai2}{2.10}{1.20}
      \node at (0.91,-1.66)  {${}_1G$};
      \node at (4.55,-1.66)  {${}_1G^*$};
      \node at (8.19,-1.66)  {${}_2G$};
      \node at (11.83,-1.66) {${}_2G^*$};
    \end{tikzpicture}%
  }
        \caption{Graph $G^{(2)}_{\sigma_1,\sigma_2,\sigma_3}(A_1,A_2)$: all twelve yellow edges are shown.}
        \label{figure-example-power-c}
    \end{subfigure}
    \caption{$k=3,m=2$, moment-graph construction (panels (b)--(c): $p=2$) with $\sigma_1 = \emptyset$, $\sigma_2 (1) = 2$ and $\sigma_3 = (12)$}
    \label{figure-example-G-power}
\end{figure}

In particular, in the graph $G^{(p)}_{\sigma_1,\ldots,\sigma_k} (A_1, \ldots, A_m)$, the number of yellow directed edges is $2p \sum_{j=1}^k (m - |D(\sigma_j)|)$ and the number of other edges is $2p \sum_{j=1}^k |D(\sigma_j)|$.

The yellow-edge identifications can also be read directly from matrix multiplication. Using the ordered row and column tuples $\alpha,\beta\in[N]^q$ from \eqref{eq-coordinate-partial}, we have
\begin{equation}\label{eq-coordinate-moment}
\Tr((YY^*)^p)
=\sum_{\substack{\alpha_1,\ldots,\alpha_p\in[N]^q\\
                 \beta_1,\ldots,\beta_p\in[N]^q}}
\prod_{r=1}^p Y_{\alpha_r;\beta_r}\,
\overline{Y_{\alpha_{r+1};\beta_r}},
\end{equation}
where we use the convention $\alpha_{p+1}=\alpha_1$.
For $q=0$, the index tuples are empty and the formula reads $|Y|^{2p}$. The common $\beta_r$ identifies the open column indices of the $r$-th copy of $Y$ with their mirrors in the same copy of $Y^*$. The common $\alpha_{r+1}$ identifies the remaining row indices with those of the next copy of $Y$; the convention $\alpha_{p+1}=\alpha_1$ closes the trace. These are precisely the yellow edges described above.

Let us point out that the graph $G^{(p)}_{\sigma_1,\ldots,\sigma_k} (A_1, \ldots, A_m)$ is a graph of a type introduced in Section \ref{sec:graph-multi leg}, if we replace the color of the yellow directed edges with the color of their vertices. Hence, we can pair the in-vertices with the out-vertices of the same rectangles in the graph $G^{(p)}_{\sigma_1,\ldots,\sigma_k} (A_1, \ldots, A_m)$ using blue directed edges, as in Section \ref{sec:graph-multi leg}.

\section{Main results} \label{sec:results}

In this paper, we obtain a result irrespective of the number of legs in the tensor model. However, for the sake of clarity, we choose to ``warm up'' with the two leg case, whose proof we give in detail, before tackling the general $k$-legs model.

\subsection{Partial trace with 2 legs}

Our first result gives upper and lower bounds for the maximum of the partial trace under operator norm constraints.

\begin{theorem} \label{Thm-main}
Let $m,N\in\bN$ and $\sigma_1,\sigma_2\in\cP([m])$. Let $M(\sigma_1,\sigma_2)$ and $C(\sigma_1,\sigma_2)$ be the quantities given in Definition~\ref{def-cycle-cut}. Then
\begin{align*}
    N^{M(\sigma_1,\sigma_2)} \le
    \max_{A_1,\ldots,A_m} \left| \left( \Tr_{\sigma_1} \otimes \Tr_{\sigma_2} \right) \left( A_1, \ldots, A_m \right) \right|
    \le N^{C(\sigma_1,\sigma_2)},
\end{align*}
where the maximum is taken among all unitary matrices $A_1, \ldots, A_m \in M_N(\bC) \otimes M_N(\bC)$.
\end{theorem}

The same bounds hold for contractions, by convexity and multilinearity. The reduction is proved in Section~\ref{sec:2legs}.

\begin{corollary} \label{Coro-main-2 leg}
For $m,N \in \bN$, for any permutations $\sigma_1,\sigma_2 \in \cP([m])$, we have
\begin{align*}
    N^{M(\sigma_1,\sigma_2)} \le
    \max_{\|A_1\|,\ldots,\|A_m\| \le 1} \left| \left( \Tr_{\sigma_1} \otimes \Tr_{\sigma_2} \right) \left( A_1, \ldots, A_m \right) \right|
    \le N^{C(\sigma_1,\sigma_2)}.
\end{align*}
\end{corollary}

We next express $C(\sigma_1,\sigma_2)$ in terms of orders of the rectangles. For a partial permutation $\sigma$, write
\[
R(\sigma)=\#\{i\in D(\sigma):\sigma(i)\le i\}.
\]
In the natural order $A_1,\ldots,A_m$, a colored edge from $A_i$ to $A_{\sigma(i)}$ is called \emph{backward} if $\sigma(i)\le i$; loops are included. More generally, for a linear order $\prec$ of the rectangles, let $F_\prec$ be the set of backward colored edges in that order, namely edges whose target is at or to the left of their source when the rectangles are drawn in that order.

\begin{proposition}[Optimal ordering]\label{Prop-cut-order}
For any graph $G$ of partial permutations $\sigma_1,\ldots,\sigma_k$,
\begin{equation}
C(\sigma_1,\ldots,\sigma_k)=\min_\prec |F_\prec|.
\end{equation}
\end{proposition}
\begin{proof}
For a graph of partial permutations, collapse each rectangle to a vertex, keeping every retained colored directed edge. The resulting directed graph has no directed cycle if and only if the graph is simple.
Indeed, a cycle with blue directed edges gives a directed closed walk in the collapsed graph, which contains a directed cycle. Conversely, any directed cycle (including loops) in the collapsed graph can be realized by pairing its in-vertex and out-vertex inside each rectangle on that cycle and completing all other blue pairings arbitrarily.

Deleting all directed edges in $F_\prec$ leaves only edges going strictly forward, so $C(\sigma_1,\ldots,\sigma_k)\le |F_\prec|$ for every order. Conversely, choose a set $F$ of $C(\sigma_1,\ldots,\sigma_k)$ edges whose removal leaves a full simple partial graph. Its collapsed graph does not have any cycles and therefore has an order $\prec$ of rectangles in which every retained edge goes strictly forward. For all backward edges in this order $\prec$ of rectangles, their corresponding directed edges in the original graph belong to $F$, giving $|F_\prec|\le C(\sigma_1,\ldots,\sigma_k)$. This proves the identity.
\end{proof}

In the natural order for two legs, $|F_\prec|=R(\sigma_1)+R(\sigma_2)$. The upper bound in Corollary~\ref{Coro-main-2 leg} therefore gives the following estimate.

\begin{corollary} \label{Coro-2 leg-backward}
For $m,N \in \bN$, for any permutations $\sigma_1,\sigma_2 \in \cP([m])$, we have
\begin{align*}
    \max_{\|A_1\|,\ldots,\|A_m\| \le 1} \left| \left( \Tr_{\sigma_1} \otimes \Tr_{\sigma_2} \right) \left( A_1, \ldots, A_m \right) \right|
    \le N^{R(\sigma_1) + R(\sigma_2)}.
\end{align*}
\end{corollary}

\begin{remark}
The backward-edge count depends on the order of the rectangles, whereas the maximum of the index contraction is invariant under simultaneous relabelling. Thus, for every $\theta\in\cP([m])$,
\[
\max_{\|A_i\|\le1}|(\Tr_{\sigma_1}\otimes\Tr_{\sigma_2})(A_1,\ldots,A_m)|
\le N^{R(\theta\sigma_1\theta^{-1})+R(\theta\sigma_2\theta^{-1})}.
\]
For example, if $\sigma_1=(164253)$ and $\sigma_2=(123456)$, then $R(\sigma_1)=4$ and $R(\sigma_2)=1$, giving $N^5$. Relabelling by $\theta=(15)(26)$ gives $\theta\sigma_1\theta^{-1}=(524613)$ and $\theta\sigma_2\theta^{-1}=(563412)$, each with two backward edges, and improves the bound to $N^4$.

By Proposition~\ref{Prop-cut-order}, optimizing over all relabellings gives exactly $N^{C(\sigma_1,\sigma_2)}$.
\end{remark}

\subsection{Partial trace with multiple legs}

Let $k \in \bN$ be the number of legs and $N \in \bN$ be the dimension of each leg. The multi-leg version of Theorem \ref{Thm-main} is as follows:

\begin{theorem} \label{Thm-main-multi}
Let $m,k,N\in\bN$ and $\sigma_1,\ldots,\sigma_k\in\cP([m])$. Let $M(\sigma_1,\ldots,\sigma_k)$ and $C(\sigma_1,\ldots,\sigma_k)$ be the quantities given in Definition~\ref{def-cycle-cut}. For $A_i\in M_N(\bC)^{\otimes k}$,
\begin{align*}
    N^{M(\sigma_1,\ldots,\sigma_k)} \le
    \max_{\|A_1\|,\ldots,\|A_m\| \le 1} \left| \left( \Tr_{\sigma_1} \otimes \ldots \otimes \Tr_{\sigma_k} \right) \left( A_1, \ldots, A_m \right) \right|
    \le N^{C(\sigma_1,\ldots,\sigma_k)}.
\end{align*}
\end{theorem}

Theorem~\ref{Thm-main-multi} also handles legs of different dimensions that are powers of a common $N$. For positive integers $a_1,\ldots,a_k$, define both
\[
\begin{split}
M(\sigma_1,\ldots,\sigma_k;a_1,\ldots,a_k)
&=M(\underbrace{\sigma_1,\ldots,\sigma_1}_{a_1},\ldots,
     \underbrace{\sigma_k,\ldots,\sigma_k}_{a_k}),\\
C(\sigma_1,\ldots,\sigma_k;a_1,\ldots,a_k)
&=C(\underbrace{\sigma_1,\ldots,\sigma_1}_{a_1},\ldots,
     \underbrace{\sigma_k,\ldots,\sigma_k}_{a_k}).
\end{split}
\]
In particular, $M(\sigma,\tau;a,b)$ and $C(\sigma,\tau;a,b)$ use $a$ copies of $\sigma$ and $b$ copies of $\tau$.

\begin{corollary} \label{Coro-main-multi-2 leg}
Let $m,N,a,b \in \bN$, and let $\sigma,\tau \in \cP([m])$ be any permutations. With $M(\sigma,\tau;a,b)$ and $C(\sigma,\tau;a,b)$ as defined above, we have
\begin{align*}
    N^{M(\sigma,\tau;a,b)} \le
    \max_{\|A_1\|,\ldots,\|A_m\| \le 1} \left| \left( \Tr_{\sigma} \otimes \Tr_{\tau} \right) \left( A_1, \ldots, A_m \right) \right|
    \le N^{C(\sigma,\tau;a,b)}.
\end{align*}
\end{corollary}

Indeed, we also allow for the case of multiple legs having different dimensions for each leg.

\begin{corollary} \label{Coro-main-multi leg}
Let $m,N,k \in \bN$, $a_1,\ldots,a_k \in \bN$, and let $\sigma_1, \ldots,\sigma_k \in \cP([m])$ be any permutations. With $M(\sigma_1,\ldots,\sigma_k;a_1,\ldots,a_k)$ and $C(\sigma_1,\ldots,\sigma_k;a_1,\ldots,a_k)$ as defined above, we have
\begin{align*}
    N^{M(\sigma_1, \ldots,\sigma_k;a_1,\ldots,a_k)} \le
    \max_{\|A_1\|,\ldots,\|A_m\| \le 1} \left| \left( \Tr_{\sigma_1} \otimes \ldots \otimes \Tr_{\sigma_k} \right) \left( A_1, \ldots, A_m \right) \right|
    \le N^{C(\sigma_1, \ldots,\sigma_k;a_1,\ldots,a_k)}.
\end{align*}
\end{corollary}

A special case admits an exact evaluation. Let $\gamma=(1\,2\,\cdots\,m)$. With sufficiently many copies of this full cycle, a permutation-tensor construction attains the backward-edge upper bound.

\begin{corollary} \label{Coro-multi-backward}
Let $m,k,N\in\bN$ with $k\ge m+1$, and let $\sigma\in\cP([m])$. Then
\begin{align*}
    \max_{\|A_1\|,\ldots,\|A_m\| \le 1} \left| \left( \Tr_{\sigma} \otimes \Tr_{\gamma} \otimes \ldots \otimes \Tr_{\gamma} \right) \left( A_1, \ldots, A_m \right) \right|
    = N^{R(\sigma)+k-1}.
\end{align*}
\end{corollary}

\begin{remark} \label{Rmk-k}
The condition $k\ge m+1$ can be weakened. For each $j\in[m-1]$ with $\sigma(j)<j$, set
\[
I_j=\{u\in\bN:\sigma(j)<u\le j\}.
\]
Partition these intervals into $K$ classes such that the intervals in each class are pairwise disjoint. If there are no such intervals, set $K=0$. Then the conclusion of Corollary~\ref{Coro-multi-backward} holds whenever $k\ge K+1$. The merging procedure that repeatedly unites two disjoint sets produces such a partition and therefore a valid sufficient condition; its output can depend on the choices of merges. Remark~\ref{Rmk-k-proof} proves the assertion.
\end{remark}


\subsection{Matrix of partial trace for partial permutations} \label{sec:matrix-multi-partial}

Let $m,k,N\in\bN$ and $\sigma_1,\ldots,\sigma_k\in\cP'([m])$. For matrices $A_i\in M_N(\bC)^{\otimes k}$, define
\[
Y=(\Tr_{\sigma_1}\otimes\cdots\otimes\Tr_{\sigma_k})(A_1,\ldots,A_m)
\]
by \eqref{eq-coordinate-partial}. Let $G_{\sigma_1,\ldots,\sigma_k}(A_1,\ldots,A_m)$ be its graph given in Section \ref{sec:graph-partial}, and put
\[
q=\sum_{j=1}^k(m-|D(\sigma_j)|).
\]
Then $Y\in M_{N^q}(\bC)$; if $q=0$, this space is $\bC$.

Recall that $M(\sigma_1,\ldots,\sigma_k)$ and $C(\sigma_1,\ldots,\sigma_k)$ are given in Definition~\ref{def-cycle-cut}, extended to partial permutations in Section~\ref{sec:graph-partial}.

Recall the permutation-tensor given in \eqref{eq-permutation-tensor}.
For a blue-edge configuration $\boldsymbol\pi$, denote the permutation-tensor by
\[
Y_{\boldsymbol\pi}
=(\Tr_{\sigma_1}\otimes\cdots\otimes\Tr_{\sigma_k})(E_{\pi_1},\ldots,E_{\pi_m}).
\]
Then $Y_{\boldsymbol\pi}=N^{M(\sigma_1,\ldots,\sigma_k;\boldsymbol\pi)}P_{\boldsymbol\pi}$, where $P_{\boldsymbol\pi}$ is the unitary permutation of the boundary tensor factors induced by the open paths; set $P_{\boldsymbol\pi}=1$ if $q=0$.

\begin{theorem}\label{Thm-main-matrix}
With the notation above, for every $p\in\bN$,
\begin{equation}
N^{2pM(\sigma_1,\ldots,\sigma_k)+q}\le
\max_{\|A_1\|,\ldots,\|A_m\|\le1}\Tr((YY^*)^p)
\le N^{2pC(\sigma_1,\ldots,\sigma_k)+q}.
\end{equation}
Moreover, for any configuration $\boldsymbol\pi_0$ with $M(\sigma_1,\ldots,\sigma_k;\boldsymbol\pi_0)=M(\sigma_1,\ldots,\sigma_k)$, it holds that
\[
Y_{\boldsymbol\pi_0}Y_{\boldsymbol\pi_0}^*=N^{2M(\sigma_1,\ldots,\sigma_k)}I_{N^q},\qquad
\Tr((Y_{\boldsymbol\pi_0}Y_{\boldsymbol\pi_0}^*)^p)=N^{2pM(\sigma_1,\ldots,\sigma_k)+q}
\quad(p\in\bN).
\]
Thus, the tuple $(A_1,\ldots,A_m) = (E_{\pi_1},\ldots,E_{\pi_m})$, which is independent of $p$, attains every lower moment bound $N^{2pM(\sigma_1,\ldots,\sigma_k)+q}$.
\end{theorem}

Since $YY^*$ is positive semi-definite, taking large moments gives the following operator norm bounds.

\begin{corollary}\label{Coro-operator norm}
Under the hypotheses of Theorem~\ref{Thm-main-matrix},
\[
N^{M(\sigma_1,\ldots,\sigma_k)} \le \max_{\|A_1\|,\ldots,\|A_m\|\le1} \|Y\|\le N^{C(\sigma_1,\ldots,\sigma_k)}.
\]
The lower bound is attained by the same permutation tensors as in Theorem~\ref{Thm-main-matrix}.
\end{corollary}

The following result is the multi-leg, operator-valued counterpart of Corollary~\ref{Coro-2 leg-backward}. It is a direct consequence of Proposition~\ref{Prop-cut-order} and Corollary~\ref{Coro-operator norm}. It will provide the upper estimates in Section~\ref{sec:application}.

\begin{corollary} \label{Coro-partial-backward}
For the graph $G_{\sigma_1,\ldots,\sigma_k}$ for partial permutations $\sigma_1,\ldots,\sigma_k$, fix a linear order $\prec$ of its rectangles, and let $F_\prec$ be its backward edges as in Proposition~\ref{Prop-cut-order}. For contractions $A_1,\ldots,A_m$,
\[
\|Y\|\le N^{C(G_{\sigma_1,\ldots,\sigma_k})}\le N^{|F_\prec|}.
\]
\end{corollary}

Opening some of the auxiliary cycles in Corollary~\ref{Coro-multi-backward} gives an exact operator norm estimate.

\begin{corollary}\label{Coro-open-chain-equality}
Let $m,N\in\bN$, $r\ge0$, and $s\ge1$ be integers, with $k=1+r+s\ge m+1$. Let $\sigma\in\cP([m])$, $\gamma=(1\,2\,\cdots\,m)$, and let $\eta$ be the partial permutation with domain $[m-1]$ and $\eta(i)=i+1$. For $m=1$, take $\eta=\emptyset$. Define
\[
Z=(\Tr_\sigma\otimes\Tr_\gamma^{\otimes r}\otimes\Tr_\eta^{\otimes s})(A_1,\ldots,A_m)
\in M_{N^s}(\bC),\qquad A_i\in M_N(\bC)^{\otimes k}.
\]
The factor $\Tr_\gamma^{\otimes r}$ is omitted when $r=0$. Then
\begin{equation}
\max_{\|A_1\|,\ldots,\|A_m\|\le1}\|Z\|=N^{R(\sigma)+r}.
\end{equation}
The maximum is attained by a tuple of permutation tensors. In particular, opening one auxiliary cycle ($s=1$) gives a matrix in $M_N(\bC)$ with maximal norm $N^{R(\sigma)+k-2}$.
\end{corollary}

\paragraph{Proof strategy.}
For Theorem~\ref{Thm-main}, write the scalar contraction as the inner product of a full simple partial graph $G'_{\sigma_1,\sigma_2}$ and its complement. For unitary coefficients, both tensors have Hilbert--Schmidt norm $N^{R_{G'_{\sigma_1,\sigma_2}}/2}$, so Cauchy--Schwarz gives $N^{R_{G'_{\sigma_1,\sigma_2}}}$. Minimizing the number $R_{G'_{\sigma_1,\sigma_2}}$ of removed colored edges gives the upper bound. For the lower bound, choose permutation tensors with $M(\sigma_1,\sigma_2)$ cycles and apply \eqref{eq-coordinate-cycles}.
For Theorem~\ref{Thm-main-matrix}, the direct evaluation in Section~\ref{sec:permutation-evaluation} expresses a partial contraction at permutation tensors as a power of $N$ times a unitary permutation; this gives simultaneous attainment of all lower moment bounds. For the upper bound, use the coordinate identity \eqref{eq-coordinate-moment}, make each copy full simple, and cut one interface of yellow edges. Taking large moments then gives the operator norm bounds.

\Needspace{6\baselineskip}
\section{Upper bound in the 2 legs case} \label{sec:upper}

In this section, we establish the upper bound for the partial trace \eqref{tr-sigma-tau}.

\subsection{Upper bound for partial trace}

For $\sigma_1,\sigma_2\in\cP([m])$, consider the graph $G_{\sigma_1,\sigma_2}(A_1,\ldots,A_m)$ defined in Section~\ref{sec:graph-2 leg}.

Recall the notions of simple and full partial graphs from Definition~\ref{def-cycle-cut}. If $G'_{\sigma_1,\sigma_2}$ is full, its complement $G''_{\sigma_1,\sigma_2}$ consists only of the removed colored edges, whose number is $R_{G'_{\sigma_1,\sigma_2}}$.

The following result is the upper bound for the partial trace \eqref{tr-sigma-tau}, and is the main result in this section.

\begin{theorem} \label{Thm-upper}
For any $m\in\bN$ and permutations $\sigma_1,\sigma_2 \in \cP([m])$, for any $N \in \bN$, we have
\begin{align*}
    \max_{A_1,\ldots,A_m} \left| \left( \Tr_{\sigma_1} \otimes \Tr_{\sigma_2} \right) \left( A_1, \ldots, A_m \right) \right|
    \le N^{C(\sigma_1,\sigma_2)}.
\end{align*}
Here, $\max_{A_1,\ldots,A_m}$ is taken over all unitary matrices $A_1,\ldots,A_m \in M_N (\bC) \otimes M_N(\bC)$.
\end{theorem}

\subsection{Proof of Theorem \ref{Thm-upper}}

We evaluate a partial graph as a tensor using the coordinate convention of Section~\ref{sec:graph}. Each retained rectangle contributes its matrix entry, and each retained colored edge contributes a Kronecker delta. An index of the removed directed edge remains free: on a retained rectangle it labels an open in-vertex or out-vertex, while on the complementary tensor of removed edges it labels the corresponding free vertex. Sum over all indices whose corresponding vertices are not open, and leave these free indices unsummed. With any fixed ordering of the free indices, define
\[
\langle X,Z\rangle=\sum_{\boldsymbol\alpha}X_{\boldsymbol\alpha}\overline{Z_{\boldsymbol\alpha}},
\qquad \|X\|_2^2=\sum_{\boldsymbol\alpha}|X_{\boldsymbol\alpha}|^2.
\]
This is the Hilbert--Schmidt pairing after grouping endpoints into row and column indices.

\begin{theorem*}[Contraction identity]
For permutations $\sigma_1,\sigma_2\in\cP([m])$ and the corresponding graph $G_{\sigma_1,\sigma_2}(A_1,\ldots,A_m)$, let $G'_{\sigma_1,\sigma_2}$ be a full partial graph and $G''_{\sigma_1,\sigma_2}$ be the complement of $G'_{\sigma_1,\sigma_2}$, with its free endpoint indices ordered to match the open vertex indices of $G'_{\sigma_1,\sigma_2}$. Then
\[
G_{\sigma_1,\sigma_2} =\langle G'_{\sigma_1,\sigma_2}, G''_{\sigma_1,\sigma_2} \rangle.
\]
\end{theorem*}
\begin{proof}
The tensor $G''_{\sigma_1,\sigma_2}$ is a product of Kronecker deltas and is real. Summing the product of the entries of $G'_{\sigma_1,\sigma_2}$ and $G''_{\sigma_1,\sigma_2}$ over their common endpoint indices restores exactly the removed identifications in \eqref{eq-coordinate-trace}. Complex conjugation on $G''_{\sigma_1,\sigma_2}$ has no effect, which gives the stated pairing.
\end{proof}

\begin{remark}
For a general partition containing rectangles on both sides, an index contraction is bilinear and equals $\langle G'_{\sigma_1,\sigma_2},\overline{G''_{\sigma_1, \sigma_2}} \rangle$ with matched indices; it need not equal $\langle G'_{\sigma_1,\sigma_2}, G''_{\sigma_1,\sigma_2} \rangle$. Only the full-rectangle case above is needed for the scalar upper bound. Figure~\ref{figure-inner product'} illustrates this full-rectangle case for the graph $G_{\sigma_1,\sigma_2}$ in Figure~\ref{figure-example'}.
\begin{figure}[ht]
    \centering
    \begin{subfigure}[t]{0.64\textwidth}
        \centering
        \fitfig{\linewidth}{108bp}{%
    \begin{tikzpicture}
      \useasboundingbox (-1.00,-1.85) rectangle (6.46,1.85);
      \grbox{A}{0}{0}{$A_1$}
      \grbox{B}{1.82}{0}{$A_2$}
      \grbox{C}{3.64}{0}{$A_3$}
      \grbox{D}{5.46}{0}{$A_4$}
      \draw[gedge] (Ao1) -- (Bi1);
      \draw[gedge] (Bo1) -- (Ci1);
      \draw[redge] (Ao2) -- (Bi2);
      \draw[redge] (Bo2) -- (Ci2);
      \draw[redge] (Co2) -- (Di2);
      \node[above=12pt,font=\small] at (Ai1) {$t_1^1$};
      \node[above=12pt,font=\small] at (Co1) {$s_3^1$};
      \node[above=12pt,font=\small] at (Di1) {$t_4^1$};
      \node[above=12pt,font=\small] at (Do1) {$s_4^1$};
      \node[below=12pt,font=\small] at (Ai2) {$t_1^2$};
      \node[below=12pt,font=\small] at (Do2) {$s_4^2$};
    \end{tikzpicture}%
  }
        \caption{$G'_{\sigma_1,\sigma_2}$: all four rectangles}
    \end{subfigure}
    \hfill
    \begin{subfigure}[t]{0.32\textwidth}
        \centering
        \fitfig{\linewidth}{108bp}{%
    \begin{tikzpicture}
      \useasboundingbox (-0.70,-1.85) rectangle (2.70,1.85);
      \node[gdot] (g31in) at (0,1.05) {};
      \node[gdot] (g31out) at (2,1.05) {};
      \node[gdot] (g44in) at (0,0) {};
      \node[gdot] (g44out) at (2,0) {};
      \node[rdot] (r41in) at (0,-1.05) {};
      \node[rdot] (r41out) at (2,-1.05) {};
      \draw[gedge] (g31out) .. controls (2,1.55) and (0,1.55) .. (g31in);
      \draw[gedge] (g44out) .. controls (2,0.50) and (0,0.50) .. (g44in);
      \draw[redge] (r41out) .. controls (2,-0.55) and (0,-0.55) .. (r41in);
      \node[left=2pt,font=\small] at (g31in) {$t_1^1$};
      \node[right=2pt,font=\small] at (g31out) {$s_3^1$};
      \node[left=2pt,font=\small] at (g44in) {$t_4^1$};
      \node[right=2pt,font=\small] at (g44out) {$s_4^1$};
      \node[left=2pt,font=\small] at (r41in) {$t_1^2$};
      \node[right=2pt,font=\small] at (r41out) {$s_4^2$};
    \end{tikzpicture}%
  }
        \caption{$G''$: removed edges}
    \end{subfigure}
    \caption{A full simple partial graph $G'_{\sigma_1,\sigma_2}$ of Figure~\ref{figure-example'} and its complement $G''_{\sigma_1,\sigma_2}$, consisting of the removed green edges $3\to1$, $4\to4$ and red edge $4\to1$. Summing over the six matching endpoint indices recovers $G_{\sigma_1,\sigma_2} = \langle G'_{\sigma_1,\sigma_2}, G''_{\sigma_1,\sigma_2}\rangle$.}
    \label{figure-inner product'}
\end{figure}
\end{remark}

Next, we introduce the properties of simple partial graphs used in the norm computation. In Lemmas~\ref{Lem-simple}--\ref{Lem-inner product}, we consider partial graphs whose retained directed edges have both endpoints on retained rectangles. Thus their free indices are the open in-vertices and out-vertices of those rectangles. This class includes every full partial graph used in the upper bound and is preserved when a rectangle and all its incident edges are removed. Edge-only complementary tensors are evaluated separately.

\begin{lemma} \label{Lem-simple-subgraph}
For any $m \in \bN$, for any permutations $\sigma_1,\sigma_2 \in \cP([m])$, we consider the graph $G_{\sigma_1,\sigma_2}$. Let $H$ be a simple partial graph of the graph $G_{\sigma_1,\sigma_2}$. For any partial graph $H'$ of $H$, it is a simple partial graph of the graph $G_{\sigma_1,\sigma_2}$.
\end{lemma}

This lemma follows directly from the definition of a simple partial graph.

\begin{lemma} \label{Lem-simple}
For any $m \in \bN$, for any permutations $\sigma_1,\sigma_2 \in \cP([m])$, we consider the graph $G_{\sigma_1,\sigma_2}$. Let $H$ be a simple partial graph of the graph $G_{\sigma_1,\sigma_2}$, such that the endpoints of all directed edges are vertices of rectangles. Then in $H$, there exists a rectangle $A_i$, such that either $A_i$ does not have out-edges or $A_i$ does not have in-edges.
\end{lemma}

\begin{proof}
We only consider the graph $H$ that contains at least one rectangle.
We prove the lemma by contradiction. Assume that all rectangles in $H$ have in-edges and out-edges. We will connect the blue directed edges in the rectangles to construct a directed cycle in the following way.

If there is a loop in $H$, then we connect the in-vertex and the out-vertex of this loop with a blue directed edge to form a directed cycle. In the following, we only consider the case where there is no loop in $H$.

We can start with a rectangle $A_{i_1}$ for $i_1 \in [m]$. By assumption, the rectangle $A_{i_1}$ has an out-edge. We denote by $A_{i_2}$ the successor of $A_{i_1}$ with respect to this out-edge, where $i_2 \in [m]$ and $i_2 \not= i_1$.
Since the rectangle $A_{i_2}$ also has an out-edge, we denote by $A_{i_3}$ the successor of $A_{i_2}$ with respect to this out-edge, where $i_3 \in [m]$ and $i_3 \not= i_2$.
We continue this procedure to obtain a sequence of indices $i_1 \not= i_2 \not= \ldots$, such that there is a directed edge from rectangle $A_{i_r}$ to rectangle $A_{i_{r+1}}$ for $r \in \bN$.
As the number of rectangles in the graph $H$ is finite, the sequence of indices $i_1 \not= i_2 \not= \ldots$ contains only finitely many different numbers. Hence, there exist integers $t < s$, such that $i_t=i_s$, and $i_t, i_{t+1}, \ldots, i_{s-1}$ are distinct.
For every $t<r<s$, we use a blue directed edge in rectangle $A_{i_r}$ to connect the in-vertex of the in-edge from $A_{i_{r-1}}$ to $A_{i_r}$ with the out-vertex of the out-edge from $A_{i_r}$ to $A_{i_{r+1}}$.
We provide an example of the blue directed edges in Figure \ref{figure-Lemma-cycle-blue}.
For the repeated rectangle $A_{i_t} = A_{i_s}$, we connect the blue directed edge from the in-vertex on $A_{i_s}$ to the out-vertex on $A_{i_t}$ to close the directed cycle $A_{i_t}\to A_{i_{t+1}}\to\cdots\to A_{i_s}=A_{i_t}$ and
complete each local pairing arbitrarily.

\begin{figure}[ht]
    \centering
    \fitfig{229.01bp}{80.00bp}{%
    \begin{tikzpicture}
      \useasboundingbox (-1.982,-1.433) rectangle (6.411,1.285);
      \grbox{A}{0}{0}{$A_{i_{r-1}}$}
      \grbox{B}{2.42}{0}{$A_{i_{r}}$}
      \grbox{C}{4.73}{0}{$A_{i_{r+1}}$}
      \draw[bedge] (Ai1) -- (Ao2);  \draw[bedge] (Ai2) -- (Ao1);
      \draw[bedge] (Bi1) -- (Bo2);  \draw[bedge] (Bi2) -- (Bo1);
      \draw[bedge] (Ci1) -- (Co1);  \draw[bedge] (Ci2) -- (Co2);
      \draw[gedge] (-1.94,0.80) .. controls (-1.44,1.38) and (-0.98,1.38) .. (Ai1);
      \uparcs[gedge]{Bo1}{Ci1}{0.58}{0.50}
      \draw[gedge] (Co1) .. controls (5.71,1.38) and (5.93,1.38) .. (6.43,0.80);
      \dnarcs[redge]{Ao2}{Bi2}{0.75}{0.40}
      \foreach \x in {-1.215,1.094,3.739,5.842}{
        \node[dts] at (\x,0.80) {\dt};
        \node[dts] at (\x,-0.80) {\dt};
      }
    \end{tikzpicture}%
  }
    \caption{Blue directed edges for cycles}
    \label{figure-Lemma-cycle-blue}
\end{figure}

Since $H$ is simple, the directed cycle constructed by these blue edges is a contradiction. This concludes the proof.
\end{proof}

In order to establish the upper bound, we need the following lemmas to compute the norm of the simple partial graph.

\begin{lemma} \label{Lem-inner product}
For any $m \in \bN$, for any permutations $\sigma_1,\sigma_2 \in \cP([m])$, we consider the graph $G_{\sigma_1,\sigma_2}$. Assume that all matrices labeling rectangles are unitary. Let $H$ be a simple partial graph of the graph $G_{\sigma_1,\sigma_2}$ such that the endpoints of all directed edges are vertices of rectangles. Let $n$ be the number of rectangles in the graph $H$, $r$ be the number of green directed edges in $H$, and $s$ be the number of red directed edges in $H$. Then we have
\begin{align*}
    \|H\|_2^2 = N^{2n-r-s}.
\end{align*}
\end{lemma}

\begin{proof}
We prove this by induction on $n$, allowing $r,s\ge0$.

We start with the case $n=1$. Note that a simple partial graph with only one rectangle must be a rectangle without edges. Thus, the graph of $H$ consists of one rectangle with no edges. In this case, we have $n=1, r=s=0$. See Figure \ref{figure-k=1-1}. The quantity $\|H\|_2^2$ is represented in Figure \ref{figure-k=1-2}. Noting that $A_1^*A_1 = I_N \otimes I_N$, we have $\|H\|_2^2 = \Tr (A_1^*A_1) =  N^2$.
\begin{figure}[ht]
    \centering
    \begin{subfigure}{0.48\textwidth}
        \centering
        \fitfig{40.80bp}{84.60bp}{%
    \begin{tikzpicture}
      \useasboundingbox (-0.600,-1.237) rectangle (0.578,1.235);
      \grbox{A}{0}{0}{$A_1$}
    \end{tikzpicture}%
  }
        \caption{Graph of $H$}
        \label{figure-k=1-1}
    \end{subfigure}
    \begin{subfigure}{0.48\textwidth}
        \centering
        \fitfig{65.00bp}{157.51bp}{%
    \begin{tikzpicture}
      \useasboundingbox (-1.108,-2.756) rectangle (1.102,2.756);
      \rgbox{U}{0}{1.52}{$A_1^{*}$}
      \grbox{L}{0}{-1.52}{$A_1$}
      \draw[gedge] (Ui2) .. controls ($(Ui2)+(-0.29,-0.50)$) and ($(Li1)+(-0.29,0.50)$) .. (Li1);
      \draw[gedge] (Lo1) .. controls ($(Lo1)+(0.29,0.50)$)  and ($(Uo2)+(0.29,-0.50)$) .. (Uo2);
      \draw[redge] (Ui1) .. controls ($(Ui1)+(-0.76,-1.40)$) and ($(Li2)+(-0.76,1.40)$) .. (Li2);
      \draw[redge] (Lo2) .. controls ($(Lo2)+(0.76,1.40)$)  and ($(Uo1)+(0.76,-1.40)$) .. (Uo1);
    \end{tikzpicture}%
  }
        \caption{Graph interpretation of $\|H\|_2^2$}
        \label{figure-k=1-2}
    \end{subfigure}
    \caption{Case $n=1$}

\end{figure}

Next, we assume that the conclusion holds for the case $n-1$. That is, for any $0\le r,s \le n-1$, for any simple graph $H'$ that has $n-1$ rectangles with $r$ green directed edges and $s$ red directed edges, the conclusion holds.
Now we consider the graph $H$ with $n$ rectangles $A_1,\ldots,A_n$. Noting that $H$ is a simple partial graph of $G_{\sigma_1,\sigma_2}$, by Lemma \ref{Lem-simple}, there exists a rectangle which either does not have out-edges or does not have in-edges. Without loss of generality, we assume that the rectangle $A_1$ does not have any in-edges. We consider the number of out-edges of $A_1$. We have the following four cases.

\bigskip
\noindent {\bf Case 1.} There is neither a green out-edge nor a red out-edge on rectangle $A_1$; therefore, rectangle $A_1$ does not have any edges. See Figure \ref{figure-case1-1} for the graph $H$ and Figure \ref{figure-case1-2} for $\|H\|_2^2$. We introduce a partial graph $H'$ which is obtained from $H$ by removing the rectangle $A_1$. See Figure \ref{figure-case1-4} for the graph $H'$ and Figure \ref{figure-case1-3} for $\|H'\|_2^2$.

\begin{figure}[ht]
    \centering
    \begin{subfigure}{0.25\textwidth}
        \centering
        \fitfig{82.25bp}{48.65bp}{%
    \begin{tikzpicture}
      \useasboundingbox (-0.600,-1.237) rectangle (3.605,1.235);
      \grbox{A}{0}{0}{$A_1$}
      \grbox{B}{1.88}{0}{$A_2$}
      \node[dts] at (3.39,0.80) {\dt};
      \node[dts] at (3.39,-0.80) {\dt};
    \end{tikzpicture}%
  }
        \caption{Graph of $H$}
        \label{figure-case1-1}
    \end{subfigure}
    \hspace{-0.5cm}
    \begin{subfigure}{0.25\textwidth}
        \centering
        \fitfig{95.55bp}{125.31bp}{%
    \begin{tikzpicture}
      \useasboundingbox (-1.259,-3.157) rectangle (3.394,3.157);
      \rgbox{Ua}{0}{1.918}{$A_1^{*}$}    \rgbox{Ub}{1.85}{1.918}{$A_2^{*}$}
      \grbox{La}{0}{-1.918}{$A_1$}       \grbox{Lb}{1.85}{-1.918}{$A_2$}
      \draw[gedge] (Uai2) .. controls ($(Uai2)+(-0.45,-0.80)$) and ($(Lai1)+(-0.45,0.80)$) .. (Lai1);
      \draw[gedge] (Lao1) .. controls ($(Lao1)+(0.45,0.80)$)  and ($(Uao2)+(0.45,-0.80)$) .. (Uao2);
      \draw[redge] (Uai1) .. controls ($(Uai1)+(-0.96,-1.70)$) and ($(Lai2)+(-0.96,1.70)$) .. (Lai2);
      \draw[redge] (Lao2) .. controls ($(Lao2)+(0.93,1.70)$)  and ($(Uao1)+(0.93,-1.70)$) .. (Uao1);
      \foreach \y in {2.718,1.118,-1.118,-2.718}{ \node[dts] at (3.22,\y) {\dt}; }
    \end{tikzpicture}%
  }
        \caption{$\|H\|_2^2$}
        \label{figure-case1-2}
    \end{subfigure}
    \hspace{-0.1cm}
    \begin{subfigure}{0.25\textwidth}
        \centering
        \fitfig{43.75bp}{123.56bp}{%
    \begin{tikzpicture}
      \useasboundingbox (-0.579,-3.158) rectangle (1.563,3.158);
      \rgbox{U}{0}{1.918}{$A_2^{*}$}
      \grbox{L}{0}{-1.918}{$A_2$}
      \node[dts] at (1.40,2.718)  {\dt};
      \node[dts] at (1.40,1.118)  {\dt};
      \node[dts] at (1.40,-1.118) {\dt};
      \node[dts] at (1.40,-2.718) {\dt};
    \end{tikzpicture}%
  }
        \caption{$\|H'\|_2^2$}
        \label{figure-case1-3}
    \end{subfigure}
    \hspace{-1cm}
    \begin{subfigure}{0.25\textwidth}
        \centering
        \fitfig{43.05bp}{48.30bp}{%
    \begin{tikzpicture}
      \useasboundingbox (-0.600,-1.237) rectangle (1.576,1.235);
      \grbox{A}{0}{0}{$A_2$}
      \node[dts] at (1.40,0.80) {\dt};
      \node[dts] at (1.40,-0.80) {\dt};
    \end{tikzpicture}%
  }
        \caption{Graph of $H'$}
        \label{figure-case1-4}
    \end{subfigure}
    \caption{Case 1}

\end{figure}

Noting that $A_1^*A_1 = I_N \otimes I_N$, we have
\begin{align} \label{eq-H-norm-case1}
    \|H\|_2^2 = \Tr \left( A_1^* A_1 \right) \|H'\|_2^2 = N^2 \|H'\|_2^2.
\end{align}
Moreover, according to Lemma \ref{Lem-simple-subgraph}, $H'$ is also a simple partial graph of $G_{\sigma_1,\sigma_2}$. The number of rectangles on the graph $H'$ is $n-1$, while the numbers of green edges and red edges are $r$ and $s$, respectively. Thus, by the induction hypothesis, we have $\|H'\|_2^2 = N^{2(n-1)-r-s}$. The proof is concluded by substituting this identity into \eqref{eq-H-norm-case1}.

\bigskip

\noindent {\bf Case 2.} There is no green out-edge, but one red out-edge on rectangle $A_1$. We denote by $A_i$ the successor of $A_1$ with respect to this red out-edge. See Figure \ref{figure-case2-1} for the graph $H$, and Figure \ref{figure-case2-2} for $\|H\|_2^2$. We introduce a partial graph $H'$ that is obtained from $H$ by removing the rectangle $A_1$ and the red out-edge associated with $A_1$. See Figure \ref{figure-case2-4} for the graph $H'$ and Figure \ref{figure-case2-3} for $\|H'\|_2^2$.

\begin{figure}[ht]
    \centering
    \begin{subfigure}{0.25\textwidth}
        \centering
        \fitfig{114.01bp}{47.40bp}{%
    \begin{tikzpicture}
      \useasboundingbox (-0.600,-1.542) rectangle (6.323,1.235);
      \grbox{A}{0}{0}{$A_1$}
      \grbox{B}{1.88}{0}{$A_2$}
      \grbox{C}{4.78}{0}{$A_i$}
      \dnarcs[redge]{Ao2}{Ci2}{0.90}{0.60}
      \node[dts] at (3.39,0.80) {\dt};
      \node[dts] at (3.39,-0.80) {\dt};
      \node[dts] at (6.07,0.80) {\dt};
      \node[dts] at (6.07,-0.80) {\dt};
    \end{tikzpicture}%
  }
        \caption{Graph of $H$}
        \label{figure-case2-1}
    \end{subfigure}
    \hspace{-0.1cm}
    \begin{subfigure}{0.25\textwidth}
        \centering
        \fitfig{109.77bp}{110.33bp}{%
    \begin{tikzpicture}
      \useasboundingbox (-1.166,-3.499) rectangle (5.776,3.459);
      \rgbox{Ua}{0}{1.918}{$A_1^{*}$}  \rgbox{Ub}{1.836}{1.918}{$A_2^{*}$}  \rgbox{Uc}{4.455}{1.918}{$A_i^{*}$}
      \grbox{La}{0}{-1.918}{$A_1$}     \grbox{Lb}{1.836}{-1.918}{$A_2$}     \grbox{Lc}{4.455}{-1.918}{$A_i$}
      \draw[gedge] (Uai2) .. controls ($(Uai2)+(-0.45,-0.80)$) and ($(Lai1)+(-0.45,0.80)$) .. (Lai1);
      \draw[gedge] (Lao1) .. controls ($(Lao1)+(0.45,0.80)$)  and ($(Uao2)+(0.45,-0.80)$) .. (Uao2);
      \draw[redge] (Uai1) .. controls ($(Uai1)+(-0.85,-1.70)$) and ($(Lai2)+(-0.85,1.70)$) .. (Lai2);
      \dnarcs[redge]{Lao2}{Lci2}{0.95}{0.70}
      \uparcs[redge]{Uci1}{Uao1}{0.90}{-0.70}
      \foreach \y in {2.718,1.118,-1.118,-2.718}{
        \node[dts] at (3.22,\y) {\dt};  \node[dts] at (5.60,\y) {\dt}; }
    \end{tikzpicture}%
  }
        \caption{$\|H\|_2^2$}
        \label{figure-case2-2}
    \end{subfigure}
    \begin{subfigure}{0.25\textwidth}
        \centering
        \fitfig{81.00bp}{107.11bp}{%
    \begin{tikzpicture}
      \useasboundingbox (-0.579,-3.158) rectangle (4.221,3.158);
      \rgbox{Ua}{0}{1.918}{$A_2^{*}$}    \rgbox{Ub}{2.93}{1.918}{$A_i^{*}$}
      \grbox{La}{0}{-1.918}{$A_2$}       \grbox{Lb}{2.93}{-1.918}{$A_i$}
      \draw[redge] (Ubi1) .. controls ($(Ubi1)+(-0.885,-1.70)$) and ($(Lbi2)+(-0.885,1.70)$) .. (Lbi2);
      \foreach \y in {2.718,1.118,-1.118,-2.718}{
        \node[dts] at (1.38,\y) {\dt};  \node[dts] at (4.05,\y) {\dt}; }
    \end{tikzpicture}%
  }
        \caption{$\|H'\|_2^2$}
        \label{figure-case2-3}
    \end{subfigure}
    \hspace{-1cm}
    \begin{subfigure}{0.25\textwidth}
        \centering
        \fitfig{70.80bp}{42.90bp}{%
    \begin{tikzpicture}
      \useasboundingbox (-0.600,-1.237) rectangle (3.674,1.235);
      \grbox{A}{0}{0}{$A_2$}
      \grbox{B}{2.48}{0}{$A_i$}
      \node[dts] at (1.39,0.80) {\dt};
      \node[dts] at (1.39,-0.80) {\dt};
      \node[dts] at (3.47,0.80) {\dt};
      \node[dts] at (3.47,-0.80) {\dt};
    \end{tikzpicture}%
  }
        \caption{Graph of $H'$}
        \label{figure-case2-4}
    \end{subfigure}
    \caption{Case 2}

\end{figure}

Since $A_1^*A_1 = I_N \otimes I_N$, we have
\begin{align} \label{eq-H-norm-case2}
    \|H\|_2^2 = \Tr \left( I_N \right) \|H'\|_2^2 = N \|H'\|_2^2.
\end{align}
Moreover, according to Lemma \ref{Lem-simple-subgraph}, $H'$ is also a simple partial graph of $G_{\sigma_1,\sigma_2}$. The number of rectangles in the partial graph $H'$ is $n-1$, and the numbers of green edges and red edges are $r$ and $s-1$, respectively. Thus, by the induction hypothesis, we have $\|H'\|_2^2 = N^{2(n-1)-r-(s-1)}$. The proof is concluded by substituting this identity into \eqref{eq-H-norm-case2}.

\bigskip
\noindent {\bf Case 3.} There is one green out-edge but no red out-edge on rectangle $A_1$. This case is similar to Case 2, and details are omitted.

\bigskip
\noindent {\bf Case 4.} There is one green out-edge and one red out-edge on rectangle $A_1$. See Figure \ref{figure-case4-1} for the graph $H$ and Figure \ref{figure-case4-2} for $\|H\|_2^2$. Let $A_i$ and $A_j$ be the successors of $A_1$ with respect to the red out-edge and the green out-edge, respectively. We introduce a partial graph $H'$ that is obtained from $H$ by removing the rectangle $A_1$ along with the associated red out-edge and green out-edge. See Figure \ref{figure-case4-4} for the graph $H'$ and Figure \ref{figure-case4-3} for $\|H'\|_2^2$.

\begin{figure}[ht]
    \begin{subfigure}{0.48\textwidth}
        \centering
        \fitfig{196.01bp}{70.40bp}{%
    \begin{tikzpicture}
      \useasboundingbox (-0.601,-1.583) rectangle (8.607,1.725);
      \grbox{A}{0}{0}{$A_1$}
      \grbox{B}{1.93}{0}{$A_2$}
      \grbox{C}{4.90}{0}{$A_i$}
      \grbox{D}{7.21}{0}{$A_j$}
      \uparcs[gedge]{Ao1}{Di1}{1.15}{1.80}
      \dnarcs[redge]{Ao2}{Ci2}{0.95}{0.75}
      \foreach \x in {3.47,6.22,8.32}{
        \node[dts] at (\x,0.80) {\dt};  \node[dts] at (\x,-0.80) {\dt}; }
    \end{tikzpicture}%
  }
        \caption{Graph of $H$}
        \label{figure-case4-1}
    \end{subfigure}
    \hfill
    \begin{subfigure}{0.48\textwidth}
        \centering
        \fitfig{193.61bp}{153.61bp}{%
    \begin{tikzpicture}
      \useasboundingbox (-1.166,-3.499) rectangle (7.335,3.425);
      \rgbox{Ua}{0}{1.918}{$A_1^{*}$}  \rgbox{Ub}{1.87}{1.918}{$A_2^{*}$}
      \rgbox{Uc}{3.99}{1.918}{$A_i^{*}$} \rgbox{Ud}{6.07}{1.918}{$A_j^{*}$}
      \grbox{La}{0}{-1.918}{$A_1$}     \grbox{Lb}{1.87}{-1.918}{$A_2$}
      \grbox{Lc}{3.99}{-1.918}{$A_i$}  \grbox{Ld}{6.07}{-1.918}{$A_j$}
      \draw[gedge] (Uai2) .. controls ($(Uai2)+(-0.45,-0.80)$) and ($(Lai1)+(-0.45,0.80)$) .. (Lai1);
      \draw[redge] (Uai1) .. controls ($(Uai1)+(-0.85,-1.70)$) and ($(Lai2)+(-0.85,1.70)$) .. (Lai2);
      \uparcs[gedge]{Lao1}{Ldi1}{1.20}{1.20}
      \dnarcs[gedge]{Udi2}{Uao2}{1.25}{-1.20}
      \dnarcs[redge]{Lao2}{Lci2}{0.95}{0.70}
      \uparcs[redge]{Uci1}{Uao1}{0.85}{-0.70}
      \foreach \y in {2.718,1.118,-1.118,-2.718}{
        \foreach \x in {2.96,5.17,7.17}{ \node[dts] at (\x,\y) {\dt}; } }
    \end{tikzpicture}%
  }
        \caption{$\|H\|_2^2$}
        \label{figure-case4-2}
    \end{subfigure}
    \\
    \begin{subfigure}{0.48\textwidth}
        \centering
        \fitfig{147.21bp}{142.81bp}{%
    \begin{tikzpicture}
      \useasboundingbox (-0.579,-3.158) rectangle (5.998,3.158);
      \rgbox{Ua}{0}{1.918}{$A_2^{*}$}  \rgbox{Ub}{2.57}{1.918}{$A_i^{*}$}
      \rgbox{Uc}{4.73}{1.918}{$A_j^{*}$}
      \grbox{La}{0}{-1.918}{$A_2$}     \grbox{Lb}{2.57}{-1.918}{$A_i$}
      \grbox{Lc}{4.73}{-1.918}{$A_j$}
      \draw[redge] (Ubi1) .. controls ($(Ubi1)+(-0.65,-1.70)$) and ($(Lbi2)+(-0.65,1.70)$) .. (Lbi2);
      \draw[gedge] (Uci2) .. controls ($(Uci2)+(-0.50,-0.80)$) and ($(Lci1)+(-0.50,0.80)$) .. (Lci1);
      \foreach \y in {2.718,1.118,-1.118,-2.718}{
        \foreach \x in {1.17,3.71,5.84}{ \node[dts] at (\x,\y) {\dt}; } }
    \end{tikzpicture}%
  }
        \caption{$\|H'\|_2^2$}
        \label{figure-case4-3}
    \end{subfigure}
    \hfill
    \begin{subfigure}{0.48\textwidth}
        \centering
        \fitfig{131.61bp}{56.40bp}{%
    \begin{tikzpicture}
      \useasboundingbox (-0.600,-1.237) rectangle (5.480,1.235);
      \grbox{A}{0}{0}{$A_2$}
      \grbox{B}{2.11}{0}{$A_i$}
      \grbox{C}{4.13}{0}{$A_j$}
      \foreach \x in {1.02,3.12,5.25}{
        \node[dts] at (\x,0.80) {\dt};  \node[dts] at (\x,-0.80) {\dt}; }
    \end{tikzpicture}%
  }
        \caption{Graph of $H'$}
        \label{figure-case4-4}
    \end{subfigure}
    \caption{Case 4}

\end{figure}

Since $A_1^*A_1 = I_N \otimes I_N$, we have
\begin{align} \label{eq-H-norm-case4}
    \|H\|_2^2 = \|H'\|_2^2.
\end{align}
Moreover, $H'$ is a simple partial graph of the graph $G_{\sigma_1,\sigma_2}$ according to Lemma \ref{Lem-simple-subgraph}. The number of rectangles on the graph $H'$ is $n-1$, and the numbers of green edges and red edges are $r-1$ and $s-1$, respectively. Thus, by the induction hypothesis, we have $\|H'\|_2^2 = N^{2(n-1)-(r-1)-(s-1)} = N^{2n-r-s}$. Together with the identity \eqref{eq-H-norm-case4}, this concludes the proof.
\end{proof}

Now we are ready to prove Theorem \ref{Thm-upper}.

\begin{proof}[Proof of Theorem~\ref{Thm-upper}]
We denote by $G_{\sigma_1,\sigma_2}$ the graph corresponding to \eqref{tr-sigma-tau}. Let $G'_{\sigma_1,\sigma_2}$ be a full simple partial graph of $G_{\sigma_1,\sigma_2}$, and $G''_{\sigma_1,\sigma_2}$ be the complement partial graph of $G'_{\sigma_1,\sigma_2}$. Then, by the Cauchy--Schwarz inequality, we have
\begin{align} \label{eq-inner product-1}
    \left| \left( \Tr_{\sigma_1} \otimes \Tr_{\sigma_2} \right) \left( A_1, \ldots, A_m \right) \right|
    = | G_{\sigma_1,\sigma_2} |
    = \left| \left\langle G'_{\sigma_1,\sigma_2}, G''_{\sigma_1,\sigma_2} \right\rangle \right|
    \le \left\|G'_{\sigma_1,\sigma_2} \right\|_2 \left\| G''_{\sigma_1,\sigma_2} \right\|_2.
\end{align}

We first handle the graph $G''_{\sigma_1,\sigma_2}$. Note that $G'_{\sigma_1,\sigma_2}$ is a full simple partial graph in $G_{\sigma_1,\sigma_2}$, $G''_{\sigma_1,\sigma_2}$ is a graph without rectangles, and the total number of green directed edges and red directed edges is $R_{G'_{\sigma_1,\sigma_2}}$. Thus, we have
\begin{align}\label{eq-norm-G2}
    \left\| G''_{\sigma_1,\sigma_2} \right\|_2^2 = N^{R_{G'_{\sigma_1,\sigma_2}}}.
\end{align}

Next, we turn to the partial graph $G'_{\sigma_1,\sigma_2}$. Note that $G'_{\sigma_1,\sigma_2}$ consists of $m$ rectangles, and the total number of green directed edges and red directed edges is $2m-R_{G'_{\sigma_1,\sigma_2}}$. By Lemma \ref{Lem-inner product}, we have
\begin{align}\label{eq-norm-G1}
    \|G'_{\sigma_1,\sigma_2}\|_2^2
    = N^{2m - \left( 2m-R_{G'_{\sigma_1,\sigma_2}} \right)}
    = N^{R_{G'_{\sigma_1,\sigma_2}}}.
\end{align}

Substituting \eqref{eq-norm-G1} and \eqref{eq-norm-G2} into \eqref{eq-inner product-1}, we obtain
\begin{align*}
    \left| \left( \Tr_{\sigma_1} \otimes \Tr_{\sigma_2} \right) \left( A_1, \ldots, A_m \right) \right|
    \le N^{R_{G'_{\sigma_1,\sigma_2}}}.
\end{align*}

The proof follows by taking the maximum over all unitary matrices $A_1,\ldots,A_m \in M_N (\bC) \otimes M_N(\bC)$, and the minimum over all full simple partial graphs $G'_{\sigma_1,\sigma_2}$ of $G_{\sigma_1,\sigma_2}$.
\end{proof}

\section{Lower bound in the 2 legs case} \label{sec:lower}

In this section, we establish the lower bound for the partial trace \eqref{tr-sigma-tau}.

\subsection{Lower bound for partial trace}

We introduce the following unitary matrix, which plays a key role in the proof.
Let $E_{ij}$ be the $N \times N$ matrix with $1$ in the $(i,j)$ entry and $0$ elsewhere. We let
\begin{align*}
    U = \sum_{i,j=1}^N E_{ij} \otimes E_{ji},
\end{align*}
then $U$ is a unitary matrix in $M_N(\bC) \otimes M_N(\bC)$.
It is actually a self-adjoint involution.

Recall that $M(\sigma_1,\sigma_2)$ is given in Definition \ref{def-cycle-cut}. The following theorem provides a lower bound for the partial trace \eqref{tr-sigma-tau}, and is the main result in this section.

\begin{theorem} \label{Thm-lower}
For $m,N \in \bN$, for any permutations $\sigma_1,\sigma_2 \in \cP([m])$, we have
\begin{align*}
    \max_{A_1,\ldots,A_m \in \{I_N \otimes I_N, U\}} \left| \left( \Tr_{\sigma_1} \otimes \Tr_{\sigma_2} \right) \left( A_1, \ldots, A_m \right) \right|
    = N^{M(\sigma_1,\sigma_2)}.
\end{align*}
\end{theorem}

\begin{proof}
For $\pi\in\cP([2])$, the permutation tensor $E_\pi$ of \eqref{eq-permutation-tensor} is $I_N\otimes I_N$ if $\pi=(1)(2)$, and $U$ if $\pi=(12)$. Thus, in the graph $G_{\sigma_1,\sigma_2}$, the two blue directed edges in each rectangle must either connect vertices of the same color or connect vertices of different colors. They correspond to $I_N\otimes I_N$ and $U$, respectively, as shown in Figure~\ref{figure-blue edges-U-I}.
\begin{figure}[ht]
    \begin{subfigure}{0.48\textwidth}
        \centering
        \fitfig{28.80bp}{56.00bp}{%
    \begin{tikzpicture}
      \useasboundingbox (-0.600,-1.237) rectangle (0.578,1.235);
      \grbox{A}{0}{0}{$A_1$}
      \draw[bedge] (Ai1) -- (Ao1);
      \draw[bedge] (Ai2) -- (Ao2);
    \end{tikzpicture}%
  }
        \caption{$A_1 = I_N \otimes I_N$}
    \end{subfigure}
    \hfill
    \begin{subfigure}{0.48\textwidth}
        \centering
        \fitfig{25.60bp}{57.20bp}{%
    \begin{tikzpicture}
      \useasboundingbox (-0.600,-1.237) rectangle (0.578,1.235);
      \grbox{A}{0}{0}{$A_1$}
      \draw[bedge] (Ai1) -- (Ao2);
      \draw[bedge] (Ai2) -- (Ao1);
    \end{tikzpicture}%
  }
        \caption{$A_1 = U$}
    \end{subfigure}
    \caption{Correspondence of blue directed edges and matrix}
    \label{figure-blue edges-U-I}
\end{figure}

For a fixed blue-edge configuration $\boldsymbol\pi=(\pi_1,\ldots,\pi_m)$ for the graph $G_{\sigma_1,\sigma_2}$, equation~\eqref{eq-coordinate-cycles} gives
\begin{equation}
(\Tr_{\sigma_1}\otimes\Tr_{\sigma_2})(E_{\pi_1},\ldots,E_{\pi_m})
=N^{M(\sigma_1,\sigma_2;\boldsymbol\pi)}.
\end{equation}
Each directed cycle in the graph $G_{\sigma_1,\sigma_2}$ contributes one independent index ranging over $[N]$. Every tuple from $\{I_N\otimes I_N,U\}^m$ is realized by such a blue-edge configuration, so its value is at most $N^{M(\sigma_1,\sigma_2)}$. Choosing a blue-edge configuration with the maximum number of cycles attains this value and proves the theorem.
\end{proof}

\section{Proof of Theorem \ref{Thm-main} and Corollary \ref{Coro-main-2 leg}} \label{sec:2legs}

In this section, we first develop the proof of Theorem~\ref{Thm-main} with the help of Theorems~\ref{Thm-upper} and~\ref{Thm-lower}.

\begin{proof}[Proof of Theorem~\ref{Thm-main}]
Since $I_N\otimes I_N$ and $U$ are unitary, Theorem~\ref{Thm-lower} gives the lower bound over all unitary tuples, and Theorem~\ref{Thm-upper} gives the upper bound. Hence
\begin{equation}
N^{M(\sigma_1,\sigma_2)}
\le \max_{A_i\text{ unitary}}|(\Tr_{\sigma_1}\otimes\Tr_{\sigma_2})(A_1,\ldots,A_m)|
\le N^{C(\sigma_1,\sigma_2)}.
\end{equation}
\end{proof}

Next, we turn to the proof of Corollary~\ref{Coro-main-2 leg}.

\begin{proof}[Proof of Corollary~\ref{Coro-main-2 leg}]
If $\|A\|\le1$, we can write $A=VP$ with $0\le P\le I$ and with the polar factor $V$ that is a unitary. Thus, the matrix $W=P+i(I-P^2)^{1/2}$ is also unitary, and
\[
A=\frac{VW+VW^*}{2}.
\]
Thus every contraction is an average of two unitaries. Expanding a multilinear scalar contraction in each matrix variable and using the triangle inequality shows that its absolute value is bounded by the maximum over unitary tuples. The reverse inequality follows because every unitary is a contraction. Apply this equality of maxima to Theorem~\ref{Thm-main}.
\end{proof}

\section{Upper and lower bounds for multiple legs} \label{sec:multi-legs}

In this section, we prove Theorem~\ref{Thm-main-multi} by extending the upper and lower bounds of Sections~\ref{sec:upper} and~\ref{sec:lower} to multiple legs.

\subsection{Upper bound}

We use the graphical interpretation of Section~\ref{sec:graph-multi leg} and the notions of simple and full partial graphs from Definition~\ref{def-cycle-cut}. For a full partial graph $G'_{\sigma_1,\ldots,\sigma_k}$, its complement $G''_{\sigma_1,\ldots,\sigma_k}$ consists of $R_{G'_{\sigma_1,\ldots,\sigma_k}}$ removed colored directed edges.

We first handle the case where $A_1,\ldots,A_m$ are unitary matrices.
For the upper bound, we follow the steps in Section \ref{sec:upper}.
Let $G'_{\sigma_1,\ldots,\sigma_k}$ be a full simple partial graph of $G_{\sigma_1,\ldots,\sigma_k}$, and $G''_{\sigma_1,\ldots,\sigma_k}$ be the complement partial graph of $G'_{\sigma_1,\ldots,\sigma_k}$.
Then, by the Cauchy--Schwarz inequality, we get
\begin{align} \label{ineq-G-multi-1}
    \left| \left( \Tr_{\sigma_1} \otimes \ldots \otimes \Tr_{\sigma_k} \right) \left( A_1, \ldots, A_m \right) \right|
    = \left| G_{\sigma_1,\ldots,\sigma_k} \right|
    &= \left| \left\langle G'_{\sigma_1,\ldots,\sigma_k}, G''_{\sigma_1,\ldots,\sigma_k} \right\rangle \right| \nonumber \\
    &\le \left\|G'_{\sigma_1,\ldots,\sigma_k} \right\|_2 \left\| G''_{\sigma_1,\ldots,\sigma_k} \right\|_2.
\end{align}
Note that $G''_{\sigma_1,\ldots,\sigma_k}$ is a graph with only edges, and the number of edges is $R_{G'_{\sigma_1,\ldots,\sigma_k}}$. We have
\begin{align} \label{eq-G''-multi}
    \left\| G''_{\sigma_1,\ldots,\sigma_k} \right\|_2^2 = N^{R_{G'_{\sigma_1,\ldots,\sigma_k}}}.
\end{align}
Next, we compute $\|G'_{\sigma_1,\ldots,\sigma_k} (A_1,\ldots,A_m)\|_2$. We introduce the following multiple-leg versions of Lemma \ref{Lem-simple-subgraph} and Lemma \ref{Lem-simple}.

\begin{lemma} \label{Lem-multi-simple-subgraph}
For any $m \in \bN$, for any permutations $\sigma_1,\ldots,\sigma_k \in \cP([m])$, we consider the graph $G_{\sigma_1,\ldots,\sigma_k}$. Let $H$ be a simple partial graph of the graph $G_{\sigma_1,\ldots,\sigma_k}$. For any partial graph $H'$ of $H$, it is a simple partial graph of the graph $G_{\sigma_1,\ldots,\sigma_k}$.
\end{lemma}

This lemma follows directly from the definition of a simple partial graph.

\begin{lemma} \label{Lem-multi-simple}
For any $m \in \bN$, for any permutations $\sigma_1,\ldots,\sigma_k \in \cP([m])$, we consider the graph $G_{\sigma_1,\ldots,\sigma_k}$. Let $H$ be a simple partial graph of the graph $G_{\sigma_1,\ldots,\sigma_k}$, such that the endpoints of all directed edges are vertices of rectangles. Then in $H$, there exists a rectangle $A_i$, such that either $A_i$ does not have out-edges or $A_i$ does not have in-edges.
\end{lemma}

\begin{proof}
The proof follows from the same argument as that of Lemma \ref{Lem-simple}, and we outline it below.

We prove by contradiction and assume that all rectangles in $H$ have both in-edges and out-edges.

If there is a loop in $H$, then we connect the in-vertex and the out-vertex of this loop with a blue directed edge to get a directed cycle.

In the following, we assume that there is no loop in $H$.
We start with a rectangle $A_{i_1}$ for $i_1 \in [m]$.
As $A_{i_1}$ has an out-edge, we can find its successor $A_{i_2}$, where $i_2 \in [m]$ and $i_2 \not= i_1$.
Similarly, $A_{i_2}$ has a successor $A_{i_3}$ for some $i_3 \not= i_2$.
We can continue this procedure to obtain a sequence of indices $i_1 \not= i_2 \not= \ldots$, such that for each $r \in \bN$, there is a directed edge from rectangle $A_{i_r}$ to $A_{i_{r+1}}$.
Thus, we can find the first index $i_s$ which coincides with some $i_t$ for $t<s$.
For each $t<r<s$, in rectangle $A_{i_r}$, we can connect the in-vertex of the directed edge from $A_{i_{r-1}}$ to $A_{i_r}$ with the out-vertex of the directed edge from $A_{i_r}$ to $A_{i_{r+1}}$ with a blue directed edge. 
For the repeated rectangle $A_{i_t}=A_{i_s}$, we connect the blue directed edge from the in-vertex on $A_{i_s}$ to the out-vertex on $A_{i_t}$.
We connect blue directed edges for other in-vertices and out-vertices arbitrarily. This gives a directed cycle $A_{i_t}\to A_{i_{t+1}}\to\cdots\to A_{i_s}=A_{i_t}$.

The existence of this directed cycle contradicts the assumption that $H$ is simple.
\end{proof}

From Lemma \ref{Lem-multi-simple-subgraph} and Lemma \ref{Lem-multi-simple}, we can deduce the multiple-leg version of Lemma \ref{Lem-inner product}, allowing us to compute $\|G'_{\sigma_1,\ldots,\sigma_k}\|_2$.

\begin{lemma} \label{Lem-multi-inner product}
For any $m \in \bN$, for any permutations $\sigma_1,\ldots,\sigma_k \in \cP([m])$, we consider the graph $G_{\sigma_1,\ldots,\sigma_k}$. Assume that all matrices labeling rectangles are unitary. Let $H$ be a simple partial graph of the graph $G_{\sigma_1,\ldots,\sigma_k}$, such that the endpoints of all directed edges are vertices of rectangles. Let $n$ be the number of rectangles and $e$ be the total number of directed edges in the graph $H$. Then we have
\begin{align*}
    \|H\|_2^2 = N^{kn-e}.
\end{align*}
\end{lemma}

\begin{proof}
The proof follows from the same argument as that of Lemma \ref{Lem-inner product}, and we sketch it below.

We argue by induction on $n$. The case $n=1$ is trivial since $H$ has to be a graph with only one rectangle and no edges. Thus, $\|H\|_2^2 = N^k$.

We assume that the conclusion holds for any simple partial graph with $n-1$ rectangles.
We consider any simple partial graph $H$ with $n$ rectangles $A_1, \ldots, A_n$.
By Lemma \ref{Lem-multi-simple}, without loss of generality, we can assume that $A_1$ does not have any in-edges. Denote by $o$ the number of out-edges of $A_1$, then as in Lemma \ref{Lem-inner product}, one can show that
\begin{align*}
    \|H\|_2^2 = \Tr(I_{N^{k-o}}) \|H'\|_2^2 = N^{k-o} \|H'\|_2^2,
\end{align*}
where $H'$ is a partial graph of $H$ obtained by removing the rectangle $A_1$ together with the directed edges associated with $A_1$. Indeed, all in-vertices of $A_1$ are open, so all row indices of $A_1$ are free. In $\|H\|_2^2$, summing those indices gives
\[
\sum_{\mathbf t}(A_1)_{\mathbf t;\mathbf s}
\overline{(A_1)_{\mathbf t;\mathbf s'}}
=\delta_{\mathbf s,\mathbf s'}.
\]
The $o$ indices corresponding to the out-vertices of $A_1$ become the corresponding free indices of $H'$, while summing the other $k-o$ indices contributes $N^{k-o}$.

Moreover, $H'$ has $n-1$ rectangles and $e-o$ edges.
By Lemma \ref{Lem-multi-simple-subgraph}, $H'$ is also a simple partial graph of $G_{\sigma_1,\ldots,\sigma_k}$, leading to $\|H'\|_2^2 = N^{k(n-1)-(e-o)}$, by induction. The proof is concluded.
\end{proof}

Recall that $G'_{\sigma_1,\ldots,\sigma_k}$ is a full simple partial graph of $G_{\sigma_1,\ldots,\sigma_k}$, and its number of edges is $km - R_{G'_{\sigma_1,\ldots,\sigma_k}}$. We apply Lemma \ref{Lem-multi-inner product} to deduce that
\begin{align} \label{eq-G'-multi}
    \left\| G'_{\sigma_1,\ldots,\sigma_k} \right\|_2^2
    = N^{km - (km - R_{G'_{\sigma_1,\ldots,\sigma_k}})}
    = N^{R_{G'_{\sigma_1,\ldots,\sigma_k}}}.
\end{align}

Substituting \eqref{eq-G''-multi} and \eqref{eq-G'-multi} into \eqref{ineq-G-multi-1}, we derive
\begin{align} \label{ineq-G-multi-2}
    \left| \left( \Tr_{\sigma_1} \otimes \ldots \otimes \Tr_{\sigma_k} \right) \left( A_1, \ldots, A_m \right) \right|
    \le N^{R_{G'_{\sigma_1,\ldots,\sigma_k}}}.
\end{align}
As \eqref{ineq-G-multi-2} holds for all unitary matrices $A_1,\ldots,A_m$ and any full simple partial graph $G'_{\sigma_1,\ldots,\sigma_k}$, we obtain the following upper bound
\begin{align} \label{ineq-multi-upper}
    \max_{A_1,\ldots,A_m} \left| \left( \Tr_{\sigma_1} \otimes \ldots \otimes \Tr_{\sigma_k} \right) \left( A_1, \ldots, A_m \right) \right|
    \le N^{C(\sigma_1,\ldots,\sigma_k)},
\end{align}
where $\max_{A_1,\ldots,A_m}$ is taken over all unitary matrices $A_1,\ldots,A_m \in M_N (\bC)^{\otimes k}$.

\subsection{Lower bound}

For any blue configuration $\boldsymbol\pi=(\pi_1,\ldots,\pi_m)$, the unitary permutation tensors $E_{\pi_i}$ of \eqref{eq-permutation-tensor} correspond precisely to the connections of blue directed edges. Equation~\eqref{eq-coordinate-cycles} therefore gives
\begin{equation}\label{eq-multi-lower}
\max_{\pi_1,\ldots,\pi_m\in\cP([k])}
\left|(\Tr_{\sigma_1}\otimes\cdots\otimes\Tr_{\sigma_k})(E_{\pi_1},\ldots,E_{\pi_m})\right|
=N^{M(\sigma_1,\ldots,\sigma_k)}.
\end{equation}
This is the lower bound over all unitary matrices $A_1,\ldots,A_m$.

\subsection{Proof of Theorem \ref{Thm-main-multi}}

\begin{proof}[Proof of Theorem~\ref{Thm-main-multi}]
As every $E_\pi$ is unitary, \eqref{eq-multi-lower} gives the lower bound $N^{M(\sigma_1,\ldots,\sigma_k)}$ for the maximum over all unitary matrices $A_1,\ldots,A_m$. Equation~\eqref{ineq-multi-upper} gives the upper bound $N^{C(\sigma_1,\ldots,\sigma_k)}$.

By the argument of the proof of Corollary~\ref{Coro-main-2 leg} in Section \ref{sec:2legs} with dimension $N^k$, one can obtain the same upper bound and lower bound for the maximum over contractions $A_1,\ldots,A_m$.
\end{proof}

\section{Proof of Theorem \ref{Thm-main-matrix}} \label{sec:matrix}

We recall the matrix $Y$ and its number $q$ of open in-vertices (equivalently open out-vertices) in Section \ref{sec:matrix-multi-partial}.
We first retain the exact combinatorial cycle count for the moment graph. We then give separate lower and upper bounds for its evaluation.

\subsection{Cycles of the moment graph}

\begin{lemma}\label{Lem-moment-cycle}
For any blue-edge configuration of $G^{(p)}_{\sigma_1,\ldots,\sigma_k}$, each directed cycle either contains no yellow edges or contains at least $2p$ yellow edges.
\end{lemma}
\begin{proof}
Only yellow edges connect different copies in the cyclic sequence
\[
{}_1G_{\sigma_1,\ldots,\sigma_k}, {}_1G^*_{\sigma_1,\ldots,\sigma_k}, {}_2G_{\sigma_1,\ldots,\sigma_k}, {}_2G^*_{\sigma_1,\ldots,\sigma_k}, \ldots,{}_pG_{\sigma_1,\ldots,\sigma_k}, {}_pG^*_{\sigma_1,\ldots,\sigma_k}.
\]
Every yellow edge goes from one copy to the next in this sequence, with ${}_pG^*_{\sigma_1,\ldots,\sigma_k}$ followed by ${}_1G_{\sigma_1,\ldots,\sigma_k}$. Thus a cycle using a yellow edge must traverse the entire cyclic sequence at least once. Its number of yellow edges is a positive multiple of $2p$.
\end{proof}

\begin{proposition}\label{Prop-moment-cycle}
For every $p\in\bN$, the maximum number of directed cycles over all blue-edge configurations of the moment graph $G^{(p)}_{\sigma_1,\ldots,\sigma_k}$ is $2pM(\sigma_1,\ldots,\sigma_k)+q$.
\end{proposition}

\begin{proof}
A cycle without yellow edges lies within one of the $2p$ copies of ${}_rG_{\sigma_1,\ldots,\sigma_k}$ or ${}_rG^*_{\sigma_1,\ldots,\sigma_k}$. There are at most $M(\sigma_1,\ldots,\sigma_k)$ such cycles in each copy, since $G^*_{\sigma_1,\ldots,\sigma_k}$ is a mirror image of $G_{\sigma_1,\ldots,\sigma_k}$. There are $2pq$ yellow edges in total, so Lemma~\ref{Lem-moment-cycle} bounds the number of cycles using yellow edges by $q$. This proves the upper bound on the number of directed cycles in $G^{(p)}_{\sigma_1,\ldots,\sigma_k}$ for arbitrary blue-edge configurations.

For equality, choose a blue-edge configuration of the graph $G_{\sigma_1,\ldots,\sigma_k}$ with $M(\sigma_1,\ldots,\sigma_k)$ cycles. Every remaining component is a directed path from an open in-vertex to an open out-vertex. These paths define a bijection between the $q$ open in-vertices and the $q$ open out-vertices. We use this configuration in each ${}_rG_{\sigma_1,\ldots,\sigma_k}$ and its mirror in each ${}_rG^*_{\sigma_1,\ldots,\sigma_k}$.

The directed cycles within one of the $2p$ copies of ${}_rG_{\sigma_1,\ldots,\sigma_k}$ or ${}_rG^*_{\sigma_1,\ldots,\sigma_k}$ contribute $2pM(\sigma_1,\ldots,\sigma_k)$. A path from an open in-vertex $a$ to an open out-vertex $b$ is followed, via a yellow edge, by the mirrored path from $b$ to $a$ in ${}_rG^*_{\sigma_1,\ldots,\sigma_k}$. The next yellow edge leads to the copy of $a$ in ${}_{r+1}G_{\sigma_1,\ldots,\sigma_k}$. Repeating this through all $p$ pairs produces one directed cycle with exactly $2p$ yellow edges. Distinct paths have disjoint vertices and give distinct cycles. Every yellow edge belongs to one of these $q$ cycles, so the total is $2pM(\sigma_1,\ldots,\sigma_k)+q$.
Figure~\ref{figure-power-path} illustrates one such cycle. If $q=0$, there are no paths or yellow edges and the same count holds.
\end{proof}

\begin{figure}[ht]
    \centering
    \fitfig{431.52bp}{96.00bp}{%
    \begin{tikzpicture}
      \useasboundingbox (-0.767,-1.496) rectangle (13.988,1.903);
      \nboxx{Ra}{0}{0}{1}{2}      \nboxx{Rb}{1.82}{0}{1}{2}
      \nboxx{Rc}{3.64}{0}{1}{2}   \nboxx{Rd}{5.46}{0}{1}{2}
      \nboxx{Re}{7.28}{0}{1}{2}   \nboxx{Rf}{9.10}{0}{1}{2}
      \nboxx{Rg}{11.38}{0}{1}{2}  \nboxx{Rh}{13.20}{0}{1}{2}
      \foreach \x/\l in {0/{A_{i_1}},1.82/{A_{i_2}},3.64/{A_{i_2}^*},5.46/{A_{i_1}^*},%
                         7.28/{A_{i_1}},9.10/{A_{i_2}},11.38/{A_{i_2}^*},13.20/{A_{i_1}^*}}{%
        \node[glab] at (\x,-0.42) {$\l$};}
      \foreach \p in {Ra,Rc,Re,Rg}{ \draw[bedge] (\p i1) -- (\p o2); }
      \foreach \q in {Rb,Rd,Rf,Rh}{ \draw[bedge] (\q i2) -- (\q o1); }
      \foreach \x in {0,3.64,7.28,11.38}{%
        \draw[kedge] (\x+0.475,-0.20) -- (\x+0.79,-0.20);
        \draw[kedge] (\x+1.00,-0.20)  -- (\x+1.345,-0.20);
        \node[dts] at (\x+0.91,0.10) {\dt};
      }
      \draw[yedge] (Rbo1) -- (Rci1);
      \draw[yedge] (Rdo1) -- (Rei1);
      \draw[yedge] (Rfo1) -- (9.95,0.65);   \draw[yedge] (10.48,0.65) -- (Rgi1);
      \node[dts] at (10.24,0.02) {\dt};
      \uparcs[yedge]{Rho1}{Rai1}{1.60}{2.30}
      \node at (0.91,-1.28)  {${}_1G_{\sigma_1,\ldots,\sigma_k}$};
      \node at (4.55,-1.28)  {${}_1G^*_{\sigma_1,\ldots,\sigma_k}$};
      \node at (8.19,-1.28)  {${}_2G_{\sigma_1,\ldots,\sigma_k}$};
      \node at (12.29,-1.28) {${}_pG^*_{\sigma_1,\ldots,\sigma_k}$};
    \end{tikzpicture}%
  }
    \caption{A path, its mirrored copies, and $2p$ yellow edges form one directed cycle in the moment graph. Only this cycle is shown.}
    \label{figure-power-path}
\end{figure}

\subsection{Simultaneous attainment of the lower bounds}\label{sec:permutation-evaluation}

For any blue-edge configuration $\boldsymbol\pi$ of the graph $G_{\sigma_1,\ldots,\sigma_k}$, we choose the permutation tensors $E_{\pi_i}$ given by \eqref{eq-permutation-tensor}, and denote the resulting matrix by $Y_{\boldsymbol\pi}$ as in Theorem~\ref{Thm-main-matrix}. The formula \eqref{eq-coordinate-partial} shows that every directed cycle contributes a factor $N$. The remaining paths identify the open row indices with the open column indices according to a bijection. Hence
\begin{align*}
Y_{\boldsymbol\pi}=N^{M(\sigma_1,\ldots,\sigma_k;\boldsymbol\pi)} P_{\boldsymbol\pi},
\end{align*}
where $P_{\boldsymbol\pi}$ is a unitary permutation of the $q$ tensor factors of $(\bC^N)^{\otimes q}$; for $q=0$, set $P_{\boldsymbol\pi}=1$.
Choosing $\boldsymbol\pi_0$ with $M(\sigma_1,\ldots,\sigma_k;\boldsymbol\pi_0)=M(\sigma_1,\ldots,\sigma_k)$ gives
\begin{equation}\label{eq-moment-lower-attained}
Y_{\boldsymbol\pi_0}Y_{\boldsymbol\pi_0}^*=N^{2M(\sigma_1,\ldots,\sigma_k)}I_{N^q},\qquad
\Tr((Y_{\boldsymbol\pi_0}Y_{\boldsymbol\pi_0}^*)^p)=N^{2pM(\sigma_1,\ldots,\sigma_k)+q}.
\end{equation}
Note that $\boldsymbol\pi_0$ depends only on $G_{\sigma_1,\ldots,\sigma_k}$, not on $p$.
We would like to remark that $E_\pi^*=E_{\pi^{-1}}$, so the adjoint matrices correspond exactly to the mirrored  blue-edge configurations in Proposition~\ref{Prop-moment-cycle}.

\subsection{Upper bound}

Choose $C(\sigma_1,\ldots,\sigma_k)$ colored edges whose removal leaves a full simple partial graph $H$ of $G_{\sigma_1,\ldots,\sigma_k}$. In each ${}_rG_{\sigma_1,\ldots,\sigma_k}$, remove the corresponding edges, and in each ${}_rG^*_{\sigma_1,\ldots,\sigma_k}$ remove their mirrors. This removes $2pC(\sigma_1,\ldots,\sigma_k)$ edges in total. Finally, remove the $q$ yellow edges from ${}_pG^*_{\sigma_1,\ldots,\sigma_k}$ to ${}_1G_{\sigma_1,\ldots,\sigma_k}$.

The remaining graph is a full simple graph. Indeed, within each copy it is $H$ or its mirror, which has no cycle for any blue-edge configuration. Between copies, the retained yellow edges follow the linear order
\[{}_1G_{\sigma_1,\ldots,\sigma_k}, {}_1G^*_{\sigma_1,\ldots,\sigma_k}, \ldots,{}_pG_{\sigma_1,\ldots,\sigma_k}, {}_pG^*_{\sigma_1,\ldots,\sigma_k},\]
and cannot return to an earlier copy. Consequently, the minimum number $C(G^{(p)}_{\sigma_1,\ldots,\sigma_k})$ of removed directed edges in the moment graph $G^{(p)}_{\sigma_1,\ldots,\sigma_k}$ to get rid of directed cycles satisfies
\begin{equation}\label{eq-moment-cut-bound}
C(G^{(p)}_{\sigma_1,\ldots,\sigma_k})\le 2pC(\sigma_1,\ldots,\sigma_k)+q.
\end{equation}
The argument can be modified without difficulty for the case $q=0$.

\subsection{Proof of Theorem~\ref{Thm-main-matrix}}

\begin{proof}[Proof of Theorem~\ref{Thm-main-matrix}]
The lower bound and simultaneous attainment follow from \eqref{eq-moment-lower-attained}. For the upper bound, use the moment identity \eqref{eq-coordinate-moment} and recolor each yellow edge by its endpoint color, as in Section~\ref{sec:graph-moment}. The graph $G^{(p)}_{\sigma_1,\ldots,\sigma_k}$ then corresponds to $k$ permutations on its $2pm$ rectangles. Theorem~\ref{Thm-main-multi} bounds the partial trace over independent contraction matrices by $N^{C(G^{(p)}_{\sigma_1,\ldots,\sigma_k})}$. The coefficient choices obtained by repeating $A_i$ and $A_i^*$ are among these independent choices, since $\|A_i^*\|=\|A_i\|\le1$. Therefore
\begin{equation}
\Tr((YY^*)^p)\le N^{C(G^{(p)}_{\sigma_1,\ldots,\sigma_k})}\le N^{2pC(\sigma_1,\ldots,\sigma_k)+q},
\end{equation}
where the last inequality is \eqref{eq-moment-cut-bound}. This proves the theorem.
\end{proof}

\section{Proof of corollaries} \label{sec:proof-coro}

We prove the remaining corollaries of Section~\ref{sec:results}. Corollary~\ref{Coro-main-2 leg} was proved in Section~\ref{sec:2legs}, and Corollary~\ref{Coro-partial-backward} was proved with its statement. The exact evaluations in Corollaries~\ref{Coro-multi-backward} and~\ref{Coro-open-chain-equality} follow from explicit blue-edge constructions.

\begin{proof}[Proof of Corollary~\ref{Coro-2 leg-backward}]
Proposition~\ref{Prop-cut-order} applied in the natural order gives
\[
C(\sigma_1,\sigma_2)\le R(\sigma_1)+R(\sigma_2).
\]
The proof is concluded by the upper bound in Corollary~\ref{Coro-main-2 leg}.
\end{proof}

\begin{proof}[Proof of Corollaries~\ref{Coro-main-multi-2 leg} and~\ref{Coro-main-multi leg}]
Identify each $M_{N^{a_j}}(\bC)$ with $M_N(\bC)^{\otimes a_j}$. Under this identification, the trace indexed by $\sigma_j$ becomes $a_j$ copies of that trace, one on each $N$-dimensional leg. The identification preserves the operator norm. Both bounds therefore follow from Theorem~\ref{Thm-main-multi} applied to the expanded list of permutations.
\end{proof}

The next lemma provides a property of the blue directed edges which will be used in the proof of Corollary~\ref{Coro-multi-backward}.

\begin{lemma}[Cycle splitting]\label{Lem-cycle-splitting}
Fix a blue-edge configuration of a graph of partial permutations with any number of legs. If two blue directed edges in one rectangle belong to the same directed cycle, exchanging their out-endpoints splits that cycle into two directed cycles. All other components are unchanged, so the total number of directed cycles increases by one.
\end{lemma}

\begin{proof}
Write the two blue edges as $u_1\to v_1$ and $u_2\to v_2$, with in-vertices $u_1,u_2$ and out-vertices $v_1,v_2$. Removing the two blue directed edges from their cycle leaves two directed paths, one from $v_1$ to $u_2$ and the other from $v_2$ to $u_1$. Replacing them by $u_1\to v_2$ and $u_2\to v_1$, and keeping all other blue directed edges, closes the two paths separately.
\end{proof}

\begin{remark}
In particular, in a blue-edge configuration which maximizes the number of cycles, no directed cycle contains two blue edges from the same rectangle.
\end{remark}

\begin{proof}[Proof of Corollary~\ref{Coro-multi-backward}]
Choosing $\sigma_1=\sigma$ and $\sigma_2=\cdots=\sigma_k=\gamma=(1\,2\,\cdots\,m)$ in Theorem \ref{Thm-main-multi} gives the upper bound
\[
\max_{\|A_i\|\le1}|(\Tr_\sigma\otimes\Tr_\gamma\otimes\cdots\otimes\Tr_\gamma)(A_1,\ldots,A_m)|
\le N^{R(\sigma)+k-1}.
\]
We construct a blue-edge configuration with exactly $R(\sigma)+k-1$ cycles, which attains this bound by \eqref{eq-coordinate-cycles}.

\bigskip
\noindent\textbf{Step 1: cycles for the backward edges of $\sigma$.}

For every fixed point $i$ of $\sigma$, in rectangle $A_i$, pair the in-vertex and out-vertex of color $col_1$ by a blue directed edge. This closes the loop into its own cycle.

For every backward edge $A_j\to A_i$ of color $col_1$, where $i=\sigma(j)<j$, assign a distinct auxiliary color $c\in\{2,\ldots,k\}$. There are at most $m$ such edges, so $k\ge m+1$ supplies enough colors. We pair the in-vertex of color $col_1$ with the out-vertex of color $col_c$ at rectangle $A_i$ with a blue directed edge. We also use a blue directed edge to pair the in-vertex of color $col_c$ with the out-vertex of color $col_1$ at rectangle $A_j$. In addition, we pair the in-vertex of color $col_c$ with the out-vertex of the same color at each $A_t$ with $i<t<j$. These blue directed edges and the forward edges of partial permutation $\sigma_c=\gamma$ from $A_i$ to $A_j$, together with the backward edge $A_j\to A_i$ of color $col_1$, provide a cycle with exactly one backward edge.

These prescribed blue directed edges are compatible: distinct backward edges of $\sigma$ have distinct in-vertices and distinct out-vertices, and we assigned distinct auxiliary colors to the non-loop backward edges. Fixed points use only their own vertices of color $col_1$ and do not interfere with these paths.
Thus, with the construction above, the directed cycles are disjoint. Figure~\ref{figure-multi-blue-R} illustrates this construction.

\begin{figure}[tbp]
    \centering
    \fitfig{181.51bp}{114.51bp}{%
    \begin{tikzpicture}
      \useasboundingbox (-0.841,-1.871) rectangle (4.482,1.883);
      \nbox{A}{0}{0}{1}{2}{$A_1$}  \nbox{B}{1.82}{0}{1}{2}{$A_2$}  \nbox{C}{3.64}{0}{1}{2}{$A_3$}
      \foreach \r in {2,3}{ \draw[kedge] (Ao\r) -- (Bi\r); \draw[kedge] (Bo\r) -- (Ci\r); }
      \draw[bedge] (Ai1) -- (Ao2);
      \draw[bedge] (Bi1) -- (Bo3);  \draw[bedge] (Bi2) -- (Bo1);
      \draw[bedge] (Ci3) -- (Co1);
      \uparcs[kedge]{Ao1}{Ci1}{1.55}{0.55}
      \uparcs[kedge]{Bo1}{Ai1}{1.15}{1.60}
      \uparcs[kedge]{Co1}{Bi1}{1.00}{1.60}
      \dnarcs[kedge]{Co2}{Ai2}{1.15}{1.35}
      \dnarcs[kedge]{Co3}{Ai3}{1.55}{1.85}
    \end{tikzpicture}%
  }
    \caption{Prescribed cycles for the two strictly backward edges of $\sigma=(321)$, using two copies of $\gamma=(123)$. Unspecified blue edges are completed in Step 2.}
    \label{figure-multi-blue-R}
\end{figure}

\bigskip
\noindent\textbf{Step 2: the remaining blue directed edges.}

At each rectangle, the same number of in-vertices and out-vertices remain unpaired. Complete their pairings with blue directed edges to maximize the total number of directed cycles subject to the blue directed edges constructed in Step 1.

\bigskip

Every cycle prescribed in Step 1 contains exactly one backward edge, of color $col_1$. The other backward edges are the $k-1$ edges of the colors $col_2,\ldots,col_k$ from $A_m$ to $A_1$. None of their out-vertices at $A_m$ was used in Step 1, because an assigned path uses auxiliary out-vertices only before its terminal rectangle $A_j$, where the auxiliary in-vertex is paired with the out-vertex of color $col_1$. Since $j\le m$, no auxiliary out-vertex at $A_m$ is used.

If two of these remaining backward edges belong to the same cycle, then the two blue directed edges of their out-vertices at $A_m$ belong to the same cycle. By Lemma~\ref{Lem-cycle-splitting}, one can swap the out-endpoints of the two blue directed edges to increase the total number of directed cycles by one. This contradicts maximality of the choices of blue directed edges in Step 2. Thus, each remaining backward edge belongs to a different cycle.

Every directed cycle contains a backward edge: along forward colored edges the rectangle index strictly increases, while blue edges stay within a rectangle. Every backward edge belongs to exactly one cycle because the completed graph consists of disjoint directed cycles. There are consequently exactly $R(\sigma)+k-1$ cycles. The proof is concluded.
\end{proof}

\begin{remark}\label{Rmk-k-proof}
We justify the refinement in Remark~\ref{Rmk-k}.
We start with the construction of blue directed edges given in the proof.
For backward edges $A_j\to A_{\sigma(j)}$ with $j<m$, we can use one auxiliary color per class of pairwise disjoint intervals $I_j=(\sigma(j),j] \cap \bN$. For an interval assigned color $col_c$, construct the blue directed edges such that the entire forward path from $A_{\sigma(j)}$ to $A_j$ consists of directed edges of color $col_c$ and blue directed edges, and close it as a directed cycle using the backward edge of color $col_1$.

For disjoint intervals assigned the same color $col_c$ with $c \ge 2$, their directed cycles use different blue directed edges and different directed edges of color $col_c$. This assertion also holds for adjacent intervals, though their paths may share the same rectangle. The directed cycles for loops of color $col_1$ are also compatible. Thus, it is enough to have $K+1$ different colors in Step 1 of the proof.

Step 2 can proceed without any changes.
Indeed, every remaining backward edge leaves rectangle $A_m$ and its out-vertex is unpaired in Step 1, noting that all assigned intervals end strictly before $m$. Complete the remaining blue pairings to maximize the number of directed cycles. The same contradiction argument shows that each remaining backward edge lies in a different cycle. This again gives exactly one backward edge per cycle and proves the bound $k\ge K+1$.
\end{remark}

\begin{proof}[Proof of Corollary~\ref{Coro-operator norm}]
For every tuple of contractions $A_1,\ldots,A_m$ and every $p\in\bN$, Theorem~\ref{Thm-main-matrix} gives
\[
\|Y\|^{2p}\le\Tr((YY^*)^p)\le N^{2pC(\sigma_1,\ldots,\sigma_k)+q}.
\]
Letting $p\to\infty$ yields
\begin{equation}\label{ineq-norm-upper}
\|Y\|\le N^{C(\sigma_1,\ldots,\sigma_k)}.
\end{equation}
For the configuration $\boldsymbol\pi_0$ in Theorem \ref{Thm-main-matrix}, one can easily see that $\|Y_{\boldsymbol\pi_0}\|=N^{M(\sigma_1,\ldots,\sigma_k)}$. Therefore
\begin{equation}\label{ineq-norm-lower}
\max_{\|A_1\|,\ldots,\|A_m\|\le1}\|Y\|\ge N^{M(\sigma_1,\ldots,\sigma_k)}.
\end{equation}
The proof is concluded by combining \eqref{ineq-norm-upper} and \eqref{ineq-norm-lower}.
\end{proof}

\begin{proof}[Proof of Corollary~\ref{Coro-open-chain-equality}]
In the natural order of rectangles, the $\sigma$-leg has $R(\sigma)$ backward edges, each of the $r$ full $\gamma$-legs has one, and the $s$ open $\eta$-legs have none. Corollary~\ref{Coro-partial-backward} gives $\|Z\|\le N^{R(\sigma)+r}$.

Next, we show that the upper bound can be attained. We first replace all $\eta$ by $\gamma$. The construction in the proof of Corollary~\ref{Coro-multi-backward} provides a blue-edge configuration with $R(\sigma)+r+s$ cycles, each containing exactly one backward colored edge. We remove the backward edges $A_m\to A_1$ of color $col_{j}$ for all $r+2 \le j \le r+1+s$. These are distinct backward edges and therefore lie on distinct cycles. Removing such directed edges turns exactly $s$ cycles into open paths, leaving $R(\sigma)+r$ cycles.

Evaluating at the permutation tensors of this blue-edge configuration, Section~\ref{sec:permutation-evaluation} gives
\[
Z=N^{R(\sigma)+r}P,
\]
where $P$ is a unitary permutation of the $s$ boundary tensor factors. Thus $\|Z\|=N^{R(\sigma)+r}$, attaining the upper bound.
\end{proof}

\section{Application to random matrix theory} \label{sec:application}

The estimates in this section are applications of the results of Section~\ref{sec:results}. Their analytic input is Corollary~\ref{Coro-partial-backward}, obtained from the operator norm bound in Corollary~\ref{Coro-operator norm}, together with the permutation-tensor evaluation in Section~\ref{sec:permutation-evaluation}. These are the multi-leg and operator-valued extensions of the scalar estimates in Theorem~\ref{Thm-main} and Corollary~\ref{Coro-2 leg-backward}.

\subsection{Uniform asymptotic freeness with matrix coefficients}

Throughout this section, $n_1,n_2$ denote matrix dimensions and $\ell$ the word length, so an alternating word has $2\ell$ coefficient rectangles. In the general notation, $m=2\ell$; in the Ginibre application $k=d_1+d_2$ and $n_j=N^{d_j}$. The letter $p$ remains reserved for moment powers. We use $\varphi$ for the free trace and $\vartheta$ for a Wick pairing index; $\sigma,\tau$ remain permutations. The partial shift is denoted by $\eta$.

Asymptotic freeness of random matrices originates in Voiculescu \cite{Voiculescu1991}; the later paper \cite{Voi1998IMRN} strengthens the uniformity needed for applications to free entropy.
A uniform scalar version, obtainable from the Weingarten expansion (see \cite[Corollaries~2.4 and~2.7]{CollinsSnady2006}), takes the following form. Let $A_1,\ldots,A_\ell,B_1,\ldots,B_\ell\in M_{n_1}(\bC)$ be deterministic contractions, let $U$ be Haar distributed in $\mathcal U_{n_1}$, and let $\widehat B_i$ denote $B_i$ in the second copy of $M_{n_1}(\bC)$ in the reduced free product of two copies with respect to the normalized trace. Writing $\varphi$ for the free product trace, for fixed $\ell$ and sufficiently large $n_1$,
\[
\left|\bE\tr(A_1UB_1U^*\cdots A_\ell UB_\ell U^*)
-\varphi(A_1\widehat B_1\cdots A_\ell\widehat B_\ell)\right|
\le K(\ell)n_1^{-2}.
\]
Here $\tr=n_1^{-1}\Tr$, and the bound is uniform over the coefficient contractions. To see this, set $|\sigma|=\ell-\#\sigma$ and $\tr_\sigma=n_1^{-\#\sigma}\Tr_\sigma$, so $|\tr_\sigma(A_1,\ldots,A_\ell)|\le1$. The Weingarten asymptotic
\[
\Wg(\rho,n_1)=n_1^{-\ell-|\rho|}\bigl(\operatorname{Mob}(\rho)+O(n_1^{-2})\bigr)
\]
with leading M\"obius coefficient $\operatorname{Mob}(\rho)$ shows that the normalized summand indexed by $\sigma\tau\rho=(1,\ldots,\ell)$ has leading power
\[
n_1^{\ell-1-(|\sigma|+|\tau|+|\rho|)}.
\]
The exponent is non-positive by the triangle inequality for permutation length, and even by parity. The terms with exponent zero give the free moment through the non-crossing moment--cumulant formula \cite[Chapter~14]{NicaSpeicher2006}; every other term, and every remainder, is $O_\ell(n_1^{-2})$. There are only finitely many terms for fixed $\ell$, proving the asserted uniform bound.

We are interested in an extension with matrix coefficients. Let $A_i',B_i'\in M_{n_1}(\bC)\otimes M_{n_2}(\bC)$ be deterministic contractions, and set
\[
\widetilde B_i=(U\otimes I_{n_2})B_i'(U^*\otimes I_{n_2}).
\]
Let $\widehat B_i$ be the image of $B_i'$ in the second matrix-algebra copy tensored with $M_{n_2}(\bC)$. Both copies share the coefficient algebra $M_{n_2}(\bC)$. The quantity to estimate is
\[
\left\|\bE[(\tr\otimes\id_{n_2})(A_1'\widetilde B_1\cdots A_\ell'\widetilde B_\ell)]
-(\varphi\otimes\id_{n_2})(A_1'\widehat B_1\cdots A_\ell'\widehat B_\ell)\right\|.
\]
This is the operator norm of a deterministic matrix in $M_{n_2}(\bC)$. The case $n_2=1$ is the scalar comparison above.

Let us discuss the relation to existing literature.
Matrix-valued estimates already occur in Haagerup--Thorbj\o rnsen \cite[Theorem~5.7]{HaagerupThorbjornsen2005}: for a self-adjoint linear pencil with bounded coefficients in $M_{n_2}(\bC)$, their operator-valued resolvent error is $O(n_2^3/n_1^2)$ away from the real spectrum, with the spectral-parameter factors displayed in that theorem. Collins--Guionnet--Parraud \cite[Theorem~1.1]{MR4445344} allow general bounded deterministic matrices in $M_{n_1}(\bC)\otimes M_{n_2}(\bC)$, including entangled ones, but estimate the fully normalized scalar trace of a smooth function; their bound has the dimension factor $n_2^2/n_1^2$. Thus neither matrix-valued output nor entanglement alone distinguishes the problem considered here.

For strong convergence with coefficients on a separate tensor factor, Pisier \cite{Pisier2014} obtains convergence for $n_2=o(n_1^{1/4})$ and a factor-$2+o(1)$ upper bound for linear Gaussian polynomials when $n_2=\exp(o(n_1))$. The norm corollary of \cite{MR4445344} reaches $n_2=o(n_1^{1/3})$. Parraud's Haar expansion \cite{ParraudHaar2023} gives scalar smooth-trace expansions in even powers of $n_1^{-1}$ and tensor-coefficient norm convergence for $n_2=o(n_1/\log^3 n_1)$. Bordenave--Collins \cite{BordenaveCollins2023} obtain Haar polynomial norm estimates with a relative error $O(n_1^{-\alpha})$ for coefficient dimensions up to $\exp(n_1^\alpha)$, for a positive exponent depending on the fixed polynomial complexity. Parraud \cite[Theorems~1.1--1.2]{ParraudTensor2024} treats operator-valued Gaussian moments: for moment order $p$, the relative $L^{2p}$ bound contains $\exp(O((p^3/n_1^2)^{1/13}))$ when $p=o(n_1^{2/3})$, in addition to a factor tending to one as $p\to\infty$. This yields strong convergence for $n_2=\exp(o(n_1^{2/3}))$ under the stated strong-convergence assumptions on the deterministic coefficients. Chen--Garza-Vargas--van Handel \cite[Corollaries~1.2 and~1.4]{ChenGarzaVargasVanHandel2026} reach $n_2=\exp(o(n_1))$ for Gaussian polynomials of bounded degree; for polynomial coefficient dimension their relative upper error is $O_{\bP}(\sqrt{\log n_1/n_1})$. These norm statements compare random polynomial norms with free norms, rather than the partial-trace expectations displayed above.

Tensor embeddings give another related setting. Collins--Yuan \cite{CollinsYuan2026} consider independent GUE matrices acting on subsets $J_i$ of $k$ tensor legs, with $|J_i|>k/2$. For fixed matrix coefficients, their linear spectral estimate has error $O(N^{-\alpha/4}(\log N)^{3/4})$, where $\alpha=\min_i(2|J_i|-k)>0$; their general estimate also records the dependence on coefficient size and norms.

Theorem~\ref{Thm-Ginibre-compare} below instead gives the explicit uniform bound $K(\ell)n_2/n_1$ for the operator norm of the difference of two partial-trace expectations, with arbitrary contraction coefficients in $M_{n_1}(\bC)\otimes M_{n_2}(\bC)$. Here $n_j=N^{d_j}$ and $d_1>d_2$. No convergence of the coefficient family or decomposition into a controlled number of elementary tensors is assumed. The proof uses exact cancellation of non-crossing Wick terms and a finite combinatorial bound for the remaining terms; it does not assert strong convergence or settle the critical regime $n_2=n_1$.

\subsection{Termwise estimates in the critical regime \texorpdfstring{$n_2=n_1$}{n2 = n1}}
When the coefficient dimension $n_2$ grows with $n_1$, estimates obtained for fixed matrix coefficients need not remain uniform. The following example illustrates a difficulty at the critical scale $n_2=n_1$: the growth of a tensor contraction can compensate for the decay of its Weingarten coefficient. We examine an individual summand of the Haar expansion to explain this issue.

We write $A_{2i-1}=A_i'$ and $A_{2i}=B_i'$. The Weingarten expansion (see \cite[Corollaries~2.4 and~2.7]{CollinsSnady2006}) gives
\begin{align*}
    & \bE \left[ (\tr \otimes \id_{n_1}) \left( A_1' (U \otimes I_{n_1}) B_1' (U^* \otimes I_{n_1}) \ldots A_\ell' (U \otimes I_{n_1}) B_\ell' (U^* \otimes I_{n_1}) \right) \right] \\
    =& n_1^{-1} \sum_{\sigma,\tau,\rho\in \cP([\ell]): \sigma\tau\rho=(1,2,\ldots ,\ell)} (\Tr_{\phi} \otimes \Tr_\eta) \left( A_1, \ldots, A_{2\ell} \right) \Wg(\rho,n_1),
\end{align*}
where $\phi \in \cP([2\ell])$ sends $2i$ to $2\sigma (i)$, and $2i+1$ to $2\tau (i)+1$ modulo $2\ell$, and where $\eta$ is a partial permutation with domain $[2\ell-1]$ sending $i$ to $i+1$. The Weingarten function $\Wg(\rho,n_1)$ is of order $O(n_1^{-\ell-|\rho|}) = O(n_1^{-2\ell+\#\rho})$, where $|\rho|=\ell-\#\rho$ is the transposition length and $\#\rho$ is the number of cycles of $\rho$.

Fix an even integer $\ell\ge2$ and take $\sigma=(\ell,\ldots,1)$ and $\tau=(1)(2)\cdots(\ell)$. Then
\[
\phi=(2\ell,2\ell-2,\ldots,2)(1)(3)\cdots(2\ell-1),
\qquad \rho=(1,\ldots,\ell)^2.
\]
We have $R(\phi)=2\ell-1$, $R(\eta)=0$, and $\#\rho=2$. Corollary~\ref{Coro-partial-backward}, with two legs of dimension $n_1$ and the natural order of the rectangles, gives
\begin{align}
\max_{\|A_1\|,\ldots,\|A_{2\ell}\|\le1}
\left\|(\Tr_\phi\otimes\Tr_\eta)(A_1,\ldots,A_{2\ell})\right\|
\le n_1^{R(\phi)+R(\eta)}=n_1^{2\ell-1}.
\end{align}
This upper bound is attained by taking $A_i$ to be the identity tensor for odd $i$ and the flip tensor for even $i$. Equivalently, the blue directed edges connect in-vertices with out-vertices of the same color in odd rectangles, and connect in-vertices with out-vertices of different colors in even rectangles. Thus, each odd rectangle gives a directed cycle along its green loop and the blue directed edge. For each $j=2,\ldots,\ell$, the green directed edge $A_{2j}\to A_{2j-2}$, together with the two red directed edges from $A_{2j-2}$ to $A_{2j-1}$ and $A_{2j-1}$ to $A_{2j}$, and the prescribed blue directed edges, also gives a cycle.
The remaining vertices form one open path. Therefore, there are $\ell+(\ell-1)=2\ell-1$ directed cycles, and the permutation-tensor evaluation in Section~\ref{sec:permutation-evaluation} gives norm $n_1^{2\ell-1}$.

Consequently, for this summand,
\begin{align*}
&\max_{\|A_1\|,\ldots,\|A_{2\ell}\|\le1}
\left\|n_1^{-1}(\Tr_\phi\otimes\Tr_\eta)(A_1,\ldots,A_{2\ell})
\,\Wg(\rho,n_1)\right\|\\
&\qquad=n_1^{2\ell-2}|\Wg(\rho,n_1)|=O(1)
\qquad(n_1\to\infty).
\end{align*}
Thus, the displayed termwise estimate gives only an order-one bound at $n_2=n_1$, and does not yield decay.

This example motivates the Gaussian model considered next. There, the Wick expansion allows the non-crossing terms to be matched exactly with the free moment expansion, leaving only crossing terms to estimate. Our bounds then give a uniform error of order $n_2/n_1$ when $n_1=N^{d_1}$, $n_2=N^{d_2}$, and $d_1>d_2$.

\subsection{The Ginibre case}

We now replace Haar conjugation by multiplication by a Ginibre matrix and its adjoint. Let $\mathsf X=\mathsf X^{(n_1)}$ have independent centered complex Gaussian entries with
\begin{align} \label{eq-covariance-Gaussian entries}
\bE[\mathsf X_{ab}\overline{\mathsf X_{cd}}]=n_1^{-1}\delta_{ac}\delta_{bd},
\qquad \bE[\mathsf X_{ab}\mathsf X_{cd}]=0.
\end{align}
Let $(\mathcal A,\varphi)$ be a tracial noncommutative probability space containing $(M_{n_1}(\bC),\tr)$ and a standard circular element $\mathsf c$ free from $M_{n_1}(\bC)$; thus $\varphi(\mathsf c\mathsf c^*)=\varphi(\mathsf c^*\mathsf c)=1$. For deterministic contractions $A_i',B_i'\in M_{n_1}(\bC)\otimes M_{n_2}(\bC)$, put
\[
\check B_i=(\mathsf X\otimes I_{n_2})B_i'(\mathsf X^*\otimes I_{n_2}),
\qquad
\bar B_i=(\mathsf c\otimes I_{n_2})B_i'(\mathsf c^*\otimes I_{n_2}).
\]

\begin{theorem}\label{Thm-Ginibre-compare}
Let $\ell,N,d_1,d_2\in\bN$, $n_1=N^{d_1}$ and $n_2=N^{d_2}$, with $d_1>d_2$. There is a constant $K(\ell)$, independent of $N,d_1,d_2$ and the deterministic coefficient contractions, such that
\[
\left\|\bE[(\tr\otimes\id_{n_2})(A_1'\check B_1\cdots A_\ell'\check B_\ell)]
-(\varphi\otimes\id_{n_2})(A_1'\bar B_1\cdots A_\ell'\bar B_\ell)\right\|
\le K(\ell)N^{-d_1+d_2}.
\]
\end{theorem}

Write $A_{2i-1}=A_i'$ and $A_{2i}=B_i'$, and identify $M_{n_1}(\bC)\otimes M_{n_2}(\bC)$ with $M_N(\bC)^{\otimes(d_1+d_2)}$. A Gaussian Wick pairing is indexed by $\vartheta\in\cP([\ell])$, pairing the $i$-th occurrence of $\mathsf X$ with the $\vartheta(i)$-th occurrence of $\mathsf X^*$. Define
\[
\mathfrak p_\vartheta=\prod_{i=1}^\ell(2i-1,2\vartheta(i)),\qquad
\gamma_{2\ell}=(1,2,\ldots,2\ell),\qquad \sigma_\vartheta=\gamma_{2\ell} \mathfrak p_\vartheta.
\]
Composition acts from right to left: $(\gamma_{2\ell}\mathfrak p_\vartheta)(x)=\gamma_{2\ell}(\mathfrak p_\vartheta(x))$. We call $\vartheta$ non-crossing or crossing according to the pairing $\mathfrak p_\vartheta$ on the naturally ordered set $[2\ell]$. Let $\eta$ be the partial permutation with domain $[2\ell-1]$ and $\eta(i)=i+1$. Set
\[
Y_\vartheta=\left(\Tr_{\sigma_\vartheta}^{\otimes d_1}\otimes\Tr_\eta^{\otimes d_2}\right)(A_1,\ldots,A_{2\ell})\in M_{n_2}(\bC).
\]
Indeed, the covariance formula \eqref{eq-covariance-Gaussian entries} identifies the indices of neighboring coefficient matrices according to
\[
\sigma_\vartheta(2i-1)=2\vartheta(i)+1\pmod{2\ell},\qquad
\sigma_\vartheta(2\vartheta(i))=2i.
\]
Each of the $\ell$ Gaussian covariances contributes $n_1^{-1}$, and the normalized outer trace contributes another $n_1^{-1}$. Thus the coordinate formula~\eqref{eq-coordinate-partial} and the Gaussian Wick rule give
\begin{equation}\label{eq-Weingarten-1}
\bE[(\tr\otimes\id_{n_2})(A_1'\check B_1\cdots A_\ell'\check B_\ell)]
=N^{-d_1(\ell+1)}\sum_{\vartheta\in\cP([\ell])}Y_\vartheta,
\end{equation}
noting that
\begin{equation}\label{eq-Ginibre-noncrossing-cycles}
\#\sigma_\vartheta = \#(\gamma_{2\ell} \mathfrak p_\vartheta)=\ell+1,
\end{equation}
if $\mathfrak p_\vartheta$ is non-crossing.

For the free expansion, the only nonzero cumulants of $\mathsf c,\mathsf c^*$ are the two mixed second cumulants, both equal to one; see \cite[Example~11.23]{NicaSpeicher2006}. The free moment-cumulant formula therefore retains precisely the non-crossing pairings. For each such pairing, the indices of the coefficient matrices are contracted along the cycles of $\gamma_{2\ell} \mathfrak p_\vartheta$, with one normalized trace per cycle. There are $\ell+1$ such cycles, as proved below. Expanding each $A_i$ into elementary tensors and applying this scalar formula leaves the coefficient factors in their original order, represented by $\eta$. Replacing the $\ell+1$ normalized traces by unnormalized traces gives
\begin{equation}\label{eq-Weingarten-2}
(\varphi\otimes\id_{n_2})(A_1'\bar B_1\cdots A_\ell'\bar B_\ell)
=N^{-d_1(\ell+1)}\sum_{\vartheta:\,\mathfrak p_\vartheta\text{ non-crossing}}Y_\vartheta.
\end{equation}
Subtracting \eqref{eq-Weingarten-2} from \eqref{eq-Weingarten-1}, the non-crossing terms cancel exactly:
\begin{equation}\label{eq-Weingarten}
\begin{aligned}
&\bE[(\tr\otimes\id_{n_2})(A_1'\check B_1\cdots A_\ell'\check B_\ell)]
-(\varphi\otimes\id_{n_2})(A_1'\bar B_1\cdots A_\ell'\bar B_\ell)
\\
&\qquad=N^{-d_1(\ell+1)}\sum_{\vartheta:\,\mathfrak p_\vartheta\text{ crossing}}Y_\vartheta.
\end{aligned}
\end{equation}

Now fix $\vartheta\in\cP([\ell])$. Then
\begin{equation}\label{eq-R(sigma)}
R(\sigma_\vartheta)=\ell+1.
\end{equation}
To see this, consider a pair $\{a,b\}$ of $\mathfrak p_\vartheta$ with $a<b$. If $b<2\ell$, then $\sigma_\vartheta(a) = \gamma_{2\ell}(b) = b+1 > a$, and $\sigma_\vartheta(b) = \gamma_{2\ell}(a) = a+1 \le b$. If $b=2\ell$, then $\sigma_\vartheta(a) = \gamma_{2\ell}(2\ell) = 1\le a$ and $\sigma_\vartheta(b) \le 2\ell$. Thus the $\ell-1$ pairs not containing $2\ell$ contribute one backward edge each, and the remaining pair contributes two, giving $R(\sigma_\vartheta)=\ell+1$.

To justify \eqref{eq-Ginibre-noncrossing-cycles}, recall that $\#$ denotes the number of cycles. For $\ell=1$ the assertion is immediate. If $\ell>1$, the non-crossing pairing $\mathfrak p_\vartheta$ must contain an adjacent pair $\{a,a+1\}$. In $\gamma_{2\ell} \mathfrak p_\vartheta$, the vertex $a+1$ is a fixed point, while $a$ belongs to another cycle. Deleting the pair $(a,a+1)$ removes this fixed point and contracts the other cycle through $a$, leaving the permutation associated with the induced non-crossing pairing on $[2\ell]\setminus\{a,a+1\}$. Relabel this set increasingly by $f(i)=i$ for $i<a$ and $f(i)=i-2$ for $i>a+1$; the induced permutation is then $\gamma_{2\ell-2}\mathfrak p^{\prime}$, where $\mathfrak p^{\prime}$ is the relabelled pairing. Thus the number of cycles decreases by one, and induction gives \eqref{eq-Ginibre-noncrossing-cycles}.

The next theorem provides an exact estimate for the maximum of the operator norm of $Y_\vartheta$ for the non-crossing case. Strictly speaking, this estimate is not needed when evaluating the difference \eqref{eq-Weingarten}. However, we include it as an algebraic confirmation that the quantities we are considering are bounded.

\begin{theorem}\label{Thm-Ginibre-1}
Let $\ell,N,d_1,d_2\in\bN$ and $\vartheta\in\cP([\ell])$, with $Y_\vartheta$ defined above and coefficient contractions $A_i\in M_N(\bC)^{\otimes(d_1+d_2)}$. If $\mathfrak p_\vartheta$ is non-crossing, then
\[
\max_{\|A_1\|,\ldots,\|A_{2\ell}\|\le1}\|Y_\vartheta\|
=N^{d_1(\ell+1)}.
\]
\end{theorem}

\begin{proof}
We consider the graph $G_{\sigma_1,\ldots,\sigma_{d_1+d_2}}$ of $Y_\vartheta$ introduced in Section \ref{sec:graph-partial}, where $\sigma_1=\ldots=\sigma_{d_1} = \sigma_\vartheta$ and $\sigma_{d_1+1} = \ldots = \sigma_{d_1+d_2} = \eta$.
By \eqref{eq-R(sigma)}, the number of backward edges in $G_{\sigma_1,\ldots,\sigma_{d_1+d_2}}$ is $d_1(\ell+1)$.
Corollary~\ref{Coro-partial-backward} with the natural order of the rectangles gives the upper bound $N^{d_1(\ell+1)}$.

To see that the upper bound can be attained, we choose $A_i=I$ for every $i$, so each blue directed edge connects the in-vertex with the out-vertex of the same color.
For $1 \le j \le d_1$, the directed edges of color $col_j$ together with the blue directed edges connecting in-vertex and out-vertex of color $col_j$ produce $\ell+1$ cycles by \eqref{eq-Ginibre-noncrossing-cycles}.
For $d_1+1 \le j \le d_1+d_2$, the directed edges of color $col_j$ together with the blue directed edges connecting in-vertex and out-vertex of color $col_j$ form a directed path.
Thus the evaluation at this blue-edge configuration is $Y_\vartheta=N^{d_1(\ell+1)}I_{n_2}$, by Section~\ref{sec:permutation-evaluation}. Thus, the lower and upper bounds agree in this particular case.
\end{proof}

Now we turn to the crossing case.
We still consider the graph $G_{\sigma_1,\ldots,\sigma_{d_1+d_2}}(A_1,\ldots,A_{2\ell})$ of $Y_\vartheta$, where $\sigma_1=\ldots=\sigma_{d_1} = \sigma_\vartheta$ and $\sigma_{d_1+1} = \ldots = \sigma_{d_1+d_2} = \eta$.
For subsets $E,F\subseteq[2\ell]$, write
\[
 e_\sigma(E,F)=\#\{x\in E:\sigma(x)\in F\}.
\]

\begin{lemma}\label{lem:crossing-order}
Assume that $\mathfrak p_\vartheta$ is crossing. Then in the graph $G_{\sigma_1,\ldots,\sigma_{d_1+d_2}}(A_1,\ldots,A_{2\ell})$, there is an order of the rectangles $A_1,\ldots,A_{2\ell}$ in which at
most $\ell$ directed edges of color $col_i (1 \le i \le d_1)$ and exactly one directed edge of color $col_j (d_1+1 \le j \le d_1+d_2)$ are backward.
\end{lemma}

\begin{proof}
As $\mathfrak p_\vartheta$ is crossing, there exist $a,b,c,d \in [2\ell]$ such that
\[
  a<c<b<d,\qquad \mathfrak p_\vartheta(a)=b,\quad \mathfrak p_\vartheta(c)=d.
\]
We order the rectangles $A_1,\ldots,A_{2\ell}$ as
\begin{align} \label{eq:crossing-order}
A_{c+1}, \ldots, A_b, A_1, \ldots, A_c, A_{b+1}, \ldots, A_{2\ell},
\end{align}
and denote this order by $\prec$. Write $R_\prec(\sigma)$ for the number of edges $A_i\to A_{\sigma(i)}$ that are backward in this order, including loops. We also consider the natural order
\[
A_1, \ldots, A_{2\ell}.
\]

We set $J_1 = [1,c], J_2 = [c+1,b], J_3=[b+1,2\ell]$.
Passing from the natural order to our order
\eqref{eq:crossing-order} changes only the relative order between $A_i$ and $A_j$ for $i \in J_1$ and $j \in J_2$.
Hence, for each $d_1+1 \le j \le d_1+d_2$, among all directed edges of color $col_j$, the one from rectangle $A_c$ to rectangle $A_{c+1}$ is the only backward edge.

In our new order, the number of backward edges of color $col_j$ for each $1 \le j \le d_1$ is
\begin{align*}
    R_\prec(\sigma_\vartheta)
    &=R(\sigma_\vartheta)+e_{\sigma_\vartheta}(J_1,J_2)-e_{\sigma_\vartheta}(J_2,J_1)\\
    &=\ell+1-\bigl(e_{\sigma_\vartheta}(J_1,J_3)-e_{\sigma_\vartheta}(J_3,J_1)\bigr),
\end{align*}
where the second equality uses
\begin{align*}
    e_{\sigma_\vartheta}(J_1,J_2) + e_{\sigma_\vartheta}(J_1,J_3) = e_{\sigma_\vartheta}(J_2,J_1) + e_{\sigma_\vartheta}(J_3,J_1).
\end{align*}
This identity holds because a permutation has equally many edges leaving and entering $J_1$.

It remains to estimate the difference $e_{\sigma_\vartheta}(J_1,J_3)-e_{\sigma_\vartheta}(J_3,J_1)$.
Since
$\sigma_\vartheta=\gamma_{2\ell} \mathfrak p_\vartheta$,
\[
 \gamma_{2\ell}^{-1}(J_3)=(J_3\setminus\{2\ell\})\cup\{b\},
 \qquad
 \gamma_{2\ell}^{-1}(J_1)=(J_1\setminus\{c\})\cup\{2\ell\}.
\]
Using that $\mathfrak p_\vartheta$ is an involution, and therefore
$e_{\mathfrak p_\vartheta}(J_1,J_3)=e_{\mathfrak p_\vartheta}(J_3,J_1)$, we establish
\begin{align*}
 e_{\sigma_\vartheta}(J_1,J_3)-e_{\sigma_\vartheta}(J_3,J_1)
 &=\mathbf 1_{\{\mathfrak p_\vartheta(b)\in J_1\}}+\mathbf 1_{\{\mathfrak p_\vartheta(c)\in J_3\}}\\
 &\quad-\mathbf 1_{\{\mathfrak p_\vartheta(2\ell)\in J_1\}}
       -\mathbf 1_{\{\mathfrak p_\vartheta(2\ell)\in J_3\}}.
\end{align*}
Noting that $\mathfrak p_\vartheta(b)=a\in J_1$, $\mathfrak p_\vartheta(c)=d\in J_3$, and $J_1,J_3$ are disjoint, we obtain that
\[
 e_{\sigma_\vartheta}(J_1,J_3)-e_{\sigma_\vartheta}(J_3,J_1)\geq1,
\]
and consequently $R_\prec(\sigma_\vartheta)\leq \ell$.
\end{proof}

\begin{theorem}\label{Thm-Ginibre-2}
Let $\ell,N,d_1,d_2\in\bN$ and $\vartheta\in\cP([\ell])$, with $Y_\vartheta$ defined above and coefficient contractions $A_i\in M_N(\bC)^{\otimes(d_1+d_2)}$. If $\mathfrak p_\vartheta$ is crossing, then
\[
\max_{\|A_1\|,\ldots,\|A_{2\ell}\|\le1}\|Y_\vartheta\|
\le N^{d_1\ell+\min\{d_1,d_2\}}.
\]
\end{theorem}

\begin{proof}
We consider the graph $G_{\sigma_1,\ldots,\sigma_{d_1+d_2}}(A_1,\ldots,A_{2\ell})$ of $Y_\vartheta$ introduced in Section \ref{sec:graph-partial}, where $\sigma_1=\ldots=\sigma_{d_1} = \sigma_\vartheta$ and $\sigma_{d_1+1} = \ldots = \sigma_{d_1+d_2} = \eta$.

For the natural order of the rectangles $A_1,\ldots,A_{2\ell}$, \eqref{eq-R(sigma)} implies that the number of backward edges is $d_1(\ell+1)$. For the order in Lemma~\ref{lem:crossing-order}, there are at most $d_1\ell+d_2$ backward edges. Applying Corollary~\ref{Coro-partial-backward} to these two orders yields
\[
\|Y_\vartheta\|\le
N^{\min\{d_1(\ell+1),d_1\ell+d_2\}}
=N^{d_1\ell+\min\{d_1,d_2\}}.
\]
\end{proof}

\begin{proof}[Proof of Theorem~\ref{Thm-Ginibre-compare}]
Let $K(\ell)$ be the number of $\vartheta\in\cP([\ell])$ for which $\mathfrak p_\vartheta$ is crossing. In particular, $K(\ell)\le \ell!$. By \eqref{eq-Weingarten}, the triangle inequality and Theorem~\ref{Thm-Ginibre-2}, we have
\begin{align*}
&\left\|\bE[(\tr\otimes\id_{n_2})(A_1'\check B_1\cdots A_\ell'\check B_\ell)]
-(\varphi\otimes\id_{n_2})(A_1'\bar B_1\cdots A_\ell'\bar B_\ell)\right\|
= \left\| N^{-d_1(\ell+1)}\sum_{\vartheta:\,\mathfrak p_\vartheta\text{ crossing}}Y_\vartheta \right\| \\
&\qquad\le N^{-d_1(\ell+1)}\sum_{\vartheta:\,\mathfrak p_\vartheta\text{ crossing}}\|Y_\vartheta\|
\le K(\ell)N^{-d_1+\min\{d_1,d_2\}}.
\end{align*}
Since $d_1>d_2$, the exponent is $-d_1+d_2$, as asserted.
\end{proof}

\section*{Statements and Declarations}

We declare that there are no data associated with the paper.
The research of the first version of this paper was made without AI assistance. We used some AI assistance to edit and improve the second version of the paper -- in particular, all diagrams of v2 were AI generated. 

\printbibliography
\end{document}